\documentclass[11pt]{amsart}

\usepackage{parskip}
\usepackage{geometry, amsmath, amssymb, amsthm, mathrsfs, setspace, comment, amscd, latexsym, mathtools}
\usepackage[all]{xy}
\usepackage[hyperindex]{hyperref}

\usepackage[abs]{overpic} 
\usepackage[dvipdfmx]{pict2e}

\usepackage{color,graphicx}

\makeatletter
\@addtoreset{equation}{section}

\makeatother

\newtheorem{theorem}{Theorem}[section]
\newtheorem{lemma}[theorem]{Lemma}
\newtheorem{corollary}[theorem]{Corollary}
\newtheorem{proposition}[theorem]{Proposition}
\theoremstyle{definition}
\newtheorem{definition}[theorem]{Definition}
\newtheorem{example}[theorem]{Example}

\newtheorem*{acknowledge}{Acknowledgments}
\newtheorem{remark}[theorem]{Remark}

\newcommand{\del}{\partial}

\newcommand{\Z}{\mathbb{Z}}
\newcommand{\R}{\mathbb{R}}
\newcommand{\C}{\mathbb{C}}
\newcommand{\CP}{\mathbb{CP}}

\newcommand{\D}{\mathbb{D}}
\newcommand{\Crit}{\textrm{Crit}}
\newcommand{\Critv}{\textrm{Critv}}
\newcommand{\Symp}{\textrm{Symp}}

\title[Lefschetz-Bott fibrations on line bundles]{Lefschetz-Bott fibrations on line bundles over symplectic manifolds}
\author[Takahiro Oba]{Takahiro Oba}
\address{Research Institute of Mathematical Sciences, Kyoto University, Kyoto 606-8502, Japan}
\email{oba@kurims.kyoto-u.ac.jp}
\thanks{This work was partially supported by JSPS KAKENHI Grant Number 18J01373}
\date{\today}
\begin{document}

\maketitle

%\tableofcontents

\begin{abstract}
We describe Lefschetz-Bott fibrations on complex line bundles over symplectic manifolds explicitly. 
As an application, we construct more than one strong symplectic filling of the link of the $A_{k}$-type singularity. 
%Moreover, we interpret this construction from the singular theoretical point of view. 
In the appendix, we show that the total space of a Lefschetz-Bott fibration over the unit disk serves as a strong symplectic filling 
of a contact manifold compatible with an open book induced by the fibration.
\end{abstract}

\section{Introduction}

Lefschetz fibrations have played an important role in the study of Stein fillings of contact manifolds. 
According to results of Akbulut and Ozbagci \cite{AO}, Loi and Piergallini \cite{LP} and Giroux and Pardon \cite{GP}, 
it is known that every Stein domain admits a Lefschetz fibration over a disk. 
Conversely, thanks to a result of Eliashberg \cite{El90} (see also \cite{CE}), the total space of a Lefschetz fibration over a disk 
admits a Stein structure. 
%Such a Lefschetz fibration gives an open book decomposition of the boundary of the total space of the fibration. In fact, 
%this open book is compatible with a contact structure on the boundary induced by the Stein structure. 
Lefschetz fibrations demonstrate their ability especially in construction of various Stein fillings of contact manifolds. 
Ozbagci and Stipsicz \cite{OS}, for example, constructed infinitely many Stein fillings of contact $3$-manifolds using 
Lefschetz fibrations. 
This result was generalized to higher dimensions by the author \cite{Oba} recently. 
Moreover, in the low-dimensional case, 
one can show uniqueness of Stein fillings of contact $3$-manifolds 
through Lefschetz fibrations combined with a result of Wendl \cite{Wen} (see \cite{PV} for example).

To see a broader class of symplectic fillings of contact manifolds via fibration-like structures, 
we need to consider other ones, not Lefschetz fibration. 
In this article, we discuss \textit{(symplectic) Lefschetz-Bott fibrations}, 
which can be seen as the complexification of Morse-Bott functions 
(see Definition \ref{def: Lefschetz-Bott} for the precise definition). 
Although these fibrations were introduced by Perutz \cite{Per} in a different context, 
it is natural to expect that they help us to study 
symplectic fillings of contact manifolds as well as Lefschetz fibrations. 
In fact, as shown in Appendix \ref{section: smoothing}, 
the total space of a symplectic Lefschetz-Bott fibration serves as a strong symplectic filling of 
a contact manifold. 
However, little is known about these fibrations, and hence 
we first find examples of symplectic manifolds that admit Lefschetz-Bott fibrations. 

\begin{theorem}\label{theorem}
Let $(M,\omega)$ be a closed symplectic manifold. 
Suppose that $[\omega/2\pi] \in H^{2}(M; \R)$ has an integral lift that is 
Poincar\'{e} dual to the homology class of a symplectic hypersurface in $(M,\omega)$. 
Then, a line bundle $L$ over $(M,\omega)$ with $c_{1}(L)=-[\omega/2\pi]$ 
admits a symplectic Lefschetz-Bott fibration over $\C$. 
%whose fibers are \red{symplectomorphic to} $M\setminus \nu_{H}$.
\end{theorem}
 
A theorem of Donaldson \cite{Do} provides tons of pairs of a symplectic manifold and its symplectic hypersurface 
satisfying the assumption of Theorem \ref{theorem}. 
Indeed, given a symplectic manifold $(M,\omega)$ with $[\omega/2\pi]\in H^{2}(M; \Z)$, 
there exists a sufficiently large $k>0$ such that 
$[k\omega/2\pi]$ is Poincar\'{e} dual to the homology class of a symplectic hypersurface in $(M, k\omega)$. 
In Section \ref{section: examples}, we will give more meaningful examples. 

We would like to emphasize that 
if the dimension of $M$ is greater than $2$, 
any line bundle over $M$ never admits a Lefschetz fibration by a topological reason (see Remark \ref{remark: Stein} for the details). 
Using the theorem, in Section \ref{section: compact fibers} we construct a Lefschetz-Bott fibration over the unit disk 
on the corresponding disk bundle, which is a not Stein in general but strong symplectic filling of the circle bundle. 
Hence, Lefschetz-Bott fibrations can capture not Stein but strong symplectic fillings of contact manifolds, 
which Lefschetz fibrations cannot do.
It is known that there are contact manifolds whose contact structures are not Stein fillable but strongly fillable. 
In Example \ref{example: cyclic}, we pick up the link of a cyclic quotient singularity as such a contact manifold and see that 
a strong symplectic filling of this contact manifold admits a Lefschetz-Bott fibration over a disk.

%Our construction follows Gay and Mark's one in \cite{GM}

We will apply the above theorem to construct various strong symplectic fillings of contact manifolds. 
In their paper, Acu and Avdek \cite{AA} (cf. \cite{AA2}) gave relations in symplectic mapping class groups of 
Milnor fibers. 
In fact, such a relation connects a resolution of an isolated singularity of a complex hypersurface with its Milnor fiber. 
This enables us to generalize a combinatorial technique of mapping class groups for low-dimensional Lefschetz fibrations. 
In Section \ref{section: fillings}, we construct more than one strong symplectic filling of 
the link of the $A_{k}$-type singularity. 
This construction is actually related to blowing-up process to obtain a resolution of the singularity (see Proposition \ref{prop: resolution}). 
Furthermore, we would like to point out a paper of Smith and Thomas \cite{ST} where they discussed a symplectic surgery along 
a Lagrangian sphere. Our operation to produce various symplectic fillings corresponds to their surgery. 
In other words, we realize their surgery as an operation for Lefschetz-Bott fibrations. See Remark \ref{remark: ST} for more details.

This article is organized as follows: 
In Section \ref{section: bundles}, we review Boothby-Wang bundles and line bundles associated to them. 
Section \ref{section: Lefschetz-Bott} is divided into five subsections. 
After reviewing fibered Dehn twists and Lefschetz-Bott fibrations in Section \ref{section: FDT} and \ref{section: def of LBF}, respectively, 
we prove the main theorem in Section \ref{section: LBF over C}. 
In Section \ref{section: compact fibers}, cutting the above fibration, we 
construct a Lefschetz-Bott fibration over a disk with compact fibers. 
We examine some examples in Section \ref{section: examples}. 
In Section \ref{section: A_k}, we discuss strong symplectic fillings of the link of 
the $A_{k}$-type singularity. 
Moreover, we interpret this result from the point of view of the singularity theory and Lefschetz-Bott 
fibrations.
To make this article self-contained, 
in Appendix \ref{section: smoothing}, 
we give an explicit description of smoothing of the corners of the total space of a Lefschetz-Bott fibration over the unit disk. 
As a consequence, we also show that the smoothed manifold is a strong symplectic filling of a contact manifold compatible with an open book induced by the fibration.

\section{Line bundles}\label{section: bundles}

\subsection{Boothby-Wang bundles and associated line bundles}

Let $(B, \omega)$ be a closed symplectic manifold with 
$[\omega/2\pi] \in H^{2}(B;\Z)$, i.e., 
the cohomology class $[\omega/2\pi] \in H^{2}(B; \R)$ lies in the image of 
the natural map $H^{2}(B; \Z) \rightarrow H^{2}(B; \R)$ induced by  
the canonical inclusion $\Z \hookrightarrow \R$. 
According to the construction due to Kobayashi \cite{Kobayashi}, 
there exists a principal $S^1$-bundle 
$p: P \rightarrow B$ with 
a connection $1$-form $\alpha \in \Omega^{1}(P)$ whose curvature form is equal to 
$d\alpha = {p}^{*}\omega$. 
We call this bundle $p: (P, \alpha) \rightarrow (B, \omega)$ the \emph{Boothby-Wang bundle} 
(or the \emph{prequantization bundle}) over $(B, \omega)$. 
In our convention, the Euler class of this bundle is $e(P)=-[\omega/2\pi]$.
Note that we identify the Lie algebra of $S^1$ with $\R$ and 
regard the connection $1$-form $\alpha$ on $P$ as an ($\R$-valued) differential $1$-form on $P$. 
Since $\omega$ is a symplectic form, 
$\alpha$ is a contact form on $P$ whose Reeb vector field $R_{\alpha}$ is the generator of the $S^1$-action on 
each fiber of $p: P \rightarrow B$.

In this article, we think two types of bundles associated to $p:P \rightarrow B$. 
Define two homomorphisms $\rho, \bar{\rho}: S^1 \cong \R/\Z \rightarrow U(1) \cong S^1$ by 
$$
	\rho(\theta) = e^{2\pi i \theta},\ \ 
	\bar{\rho}(\theta) =e^{-2\pi i \theta}. 
$$
Consider the associated bundle $P\times_{\rho} \C \rightarrow B$, that is, 
the quotient space of $P \times \C$ with the right $S^1$-action defined by 
$$
	(x,z) \cdot \theta = (x \cdot \theta, \rho(-\theta)z) = (x \cdot \theta, e^{-2\pi i \theta}z) 
$$
for $(x,z) \in P \times \C$ and $\theta \in S^1$. 
By definition, the first Chern class $c_{1}(P\times_{\rho} \C)$ of this bundle agrees with $e(P)$.
One can equip the total space with a symplectic structure as follows. 
Let 
$${p}^{*}\omega + d(r^2d\theta) + d(r^2\alpha) = d((1+r^2)(\alpha+d\theta))$$ 
be a $2$-form on $P \times \C$, 
where $(r,\theta)$ are polar coordinates of $\C$. 
This $2$-form is $S^1$-invariant and horizontal, and in particular its kernel at each point $(x,z)$ is spanned by 
the vector $(R_{\alpha}-\del_{\theta})|_{(x,z)}$ (see \cite[Lemma IV. 17]{Nied} for example). 
%Indeed, 
%$$
%	\{ d((1+r^2)(\alpha+d\theta)) \}^n = n(1+r^2)^{n-1}(d\alpha)^{n-1} \wedge dr^2 \wedge (d\theta + \alpha),
%$$
%where $\dim B = 2(n-1)$ (see \cite[Lemma IV. 17]{Nied}).
Therefore, the $2$-form on $P \times \C$ descends to the symplectic form $\omega_{\alpha}$ on $P \times_{\rho} \C$. 

Next, consider the other associated bundle, namely, $P \times_{\bar{\rho}} \C \rightarrow B$. 
The total space is obtained by dividing $P \times \C$ by the $S^1$-action given by 
$$
	(x,z) \cdot \theta = (x \cdot \theta, \bar{\rho}(-\theta)z) = (x \cdot \theta, e^{2\pi i \theta}z).
$$
This construction shows that this bundle $P\times_{\bar{\rho}} \C$ is dual to the previous bundle $P \times_{\rho} \C$, and 
$c_{1}(P\times_{\bar{\rho}} \C) = -e(P)$. 
Similar to the above argument, we can equip a symplectic structure on the disk bundle $P \times_{\bar{\rho}} \mathring{\D}$, 
where $\mathring{\D}$ is the open unit disk in $\C$. 
Define a $2$-form on $P \times \mathring{\D}$ by 
$$
	{p}^{*}\omega + d(r^2d\theta) - d(r^2\alpha) = d((1-r^2)(\alpha-d\theta)). 
$$
A straightforward computation shows that this $2$-form is $S^1$-invariant and horizontal, and its kernel at each point $(x,z)$ is spanned by 
the vector $(R_{\alpha}+\del_{\theta})|_{(x,z)}$. 
%because we have 
%$$
%	\{ d((1-r^2)(\alpha-d\theta)) \}^n = n(1-r^2)^{n-1}(d\alpha)^{n-1} \wedge dr^2 \wedge (d\theta - \alpha)
%$$ 
%similarly to the above. 
Hence, the $2$-form defines the symplectic form $\omega_{\bar{\alpha}}$ on $P \times_{\bar{\rho}} \mathring{\D}$. 
Furthermore, it is 
an exact form on the complement of the zero-section $B_{0} \subset P \times_{\bar{\rho}} \mathring{\D}$, that is, 
the $1$-form $\lambda_{\bar{\alpha}}=(1-r^2)(\alpha-d\theta)$ on $(P \times_{\bar{\rho}} \mathring{\D}) \setminus B_{0}$ satisfies 
$d\lambda_{{\bar{\alpha}}} = \omega_{\bar{\alpha}}|_{(P \times_{\bar{\rho}} \mathring{\D}) \setminus B_{0}}$. 
Let 
$$
	X_{\bar{\alpha}} = -\frac{1-r^2}{2r}\del_{r}
$$
be a vector field defined in $(P \times_{\bar{\rho}} \mathring{\D}) \setminus B_{0}$. 
Then, we have $\iota_{X_{\bar{\alpha}}} \omega_{\bar{\alpha}}=\lambda_{\bar{\alpha}}$. 
This shows that $X_{\bar{\alpha}}$ is $\omega_{\bar{\alpha}}$-dual to $\lambda_{\bar{\alpha}}$, which is 
called the \emph{Liouville vector field} of $\lambda_{\bar{\alpha}}$. 
%For $\delta$ be a number with $0< \delta < 1$. 
%The total space of the bundle $P\times_{\bar{\rho}} \D(\delta) \rightarrow B$ has a concave boundary. 

%\begin{table}[htb]
%  \begin{tabular}{|c|c|c|c|c|c|}\hline
%   			& connection on $P$ & curvature on $P$ & connection on $L$ & curvature on $L$ & Chern class\\ \hline
%    $i\R$-value & $A$ & $\Omega$ & $\nabla$ & $F^{\nabla}=-i\Omega$ & $-[\Omega/2\pi i]$ \\ \hline
%    $\R$-value & $\alpha=-iA$  & $\omega= -i\Omega$; $d\alpha=p^{*}\omega$& $\nabla$ & \red{$\omega$?} & $-[\omega/2\pi]$ \\ \hline
%  \end{tabular}
%\end{table}

\subsection{Decompositions of symplectic manifolds}\label{section: decomposition}

Let $(M, \omega)$ be a closed symplectic manifold with $[\omega/2\pi] \in H^{2}(M; \Z)$. 
Suppose that the cohomology class $[\omega/2\pi]$ is Poincar\'{e} dual to the homology class of 
a symplectic hypersurface $H \subset M$. 
Set $\omega_{H} = \omega|_{TH}$ and 
let $p_{H}: (P, \alpha) \rightarrow (H, \omega_{H})$ be the Boothby-Wang bundle over $(H, \omega_{H})$.   
It is easy to check that the symplectic form $\omega$ is exact on the complement $M \setminus H$. 
Thus, 
there is a primitive $1$-form $\lambda$ of $\omega$ on $M \setminus H$, i.e., $d\lambda=\omega|_{M \setminus H}$. 
By Weinstein's tubular neighborhood theorem coupled with an argument of Diogo and Lisi \cite[Lemma 2.2]{Diogo-Lisi}, 
there exists some $0<\delta<1$, a primitive $1$-form $\lambda$ of $\omega$ on $M \setminus H$ and a symplectic embedding 
$$
\varphi_{\nu}: (P \times_{\bar{\rho}} \D(\delta), \omega_{\bar{\alpha}}) \rightarrow (M, \omega)
$$ 
such that ${\varphi_{\nu}}^{*}\lambda = \lambda_{\bar{\alpha}}$ and the zero-section under $\varphi_{\nu}$ coincides with $H$, where 
$\D(\delta)$ is the closed disk in $\C$ of radius $\delta$. 
%\cyan{The latter condition ensures that the Liouville vector field $X_{\alpha}$ is sent to 
%the one $X$ that is $\omega$-dual to $\lambda$.}
Set $\nu_{M}(H) = \varphi_{\nu}(P \times_{\bar{\rho}} \D(\delta))$ and $V= M \setminus \mathring{\nu}_{M}(H)$, 
where $\mathring{\nu}_{M}(H)$ is the interior of $\nu_{M}(H)$. 
Endow $M \setminus H$ with the Liouville vector field $X$ of $\lambda$. 
Since $\varphi_{\nu}$ pulls back $\lambda$ to $\lambda_{\bar{\alpha}}$, 
its push-forward map sends $X_{\bar{\alpha}}$ to $X$. 

For future use, we define an annulus bundle associated to $P$. 
Take $\delta' \in \R$ such that $0< \delta < \delta' <1$, and embed 
$(P \times_{\bar{\rho}} D^{2}(\delta'), \omega_{\bar{\alpha}})$ into $(M, \omega)$ symplectically 
by a canonical extension $\hat{\varphi}_{\nu}$ of the previous embedding $\varphi_{\nu}$. 
%Here we require that ${\hat{\varphi}_{\nu}}^{*}\lambda = \lambda_{\bar{\alpha}}$. 
Set $A(\delta, \delta')= \{z \in \C \, | \, \delta \leq |z| \leq \delta' \}$, which is diffeomorphic to an annulus. 
Our desired annulus bundle is 
$$
	P \times_{\bar{\rho}} A(\delta, \delta') \rightarrow H. 
$$
Note that the image $\hat{\varphi}_{\nu}(P \times_{\bar{\rho}} A(\delta, \delta'))$ lies in $V$.

\subsection{Line bundles over symplectic manifolds}\label{section: line bundles}

Let $(M, \omega)$ and $H$ be a closed symplectic manifold and its symplectic hypersurface as above, respectively. 
We assume that $M$ is decomposed as $M= V \cup \nu_{M}(H)$ as in the previous subsection. 
Here we construct a principal $S^1$-bundle over $M$ based on \cite{CDvK} 
and its associated line bundle over each piece of the decomposition, and then 
we glue two line bundles together to obtain a line bundle over $M$. 

Let $p_{V}: V \times S^1 \rightarrow V$ be the projection to the first factor. 
It can be seen as a principal $S^1$-bundle, and its connection $1$-form is 
$$
	\alpha_{V} = \lambda + d\theta_{1},
$$ 
where $\theta_{1}$ is the standard coordinate of $S^1$. 
Consider the associated line bundle $(V\times S^1) \times_{\rho} \C \rightarrow V$ endowed with 
the symplectic structure $\omega_{\alpha_{V}}$ on the total space defined by 
$$
	\omega_{\alpha_{V}}= 
	{p_{V}}^{*}(d\lambda) + d(r_{2}^2d\theta_{2}) +d(r_{2}^{2}\alpha_{V}) = d((1+r_{2}^{2})(\lambda+d\theta_{1}+d\theta_{2})), 
$$
where $(r_{2}, \theta_{2})$ are polar coordinates of $\C$. 
In fact, this line bundle is isomorphic to the product bundle $\Pi_{V}: V \times \C \rightarrow V$ by 
$$
	\Psi_{V}: (V\times S^1) \times_{\rho} \C \rightarrow V \times \C, \  
	[x,\theta_{1},(r_{2}, \theta_{2})] \rightarrow (x,(r_{2}, \theta_{1}+\theta_{2})). 
$$
Note that 
$S^{1}$ acts on $V \times S^1$ by 
$$
 (x,\theta_{1}) \cdot \theta =  (x,\theta_{1}+\theta). 
$$
Letting $\Omega_{V} = d((1+r^2)(\alpha_{V}+d\theta)) \in \Omega^{2}(V\times \C)$, 
since we have $$\Psi_{V}^{*}((1+r^{2})(\lambda+d\theta)) = (1+r_{2}^{2})(\lambda+d\theta_{1}+d\theta_{2}),$$
the map $\Psi_{V}: ((V \times S^1) \times_{\rho} \C, \omega_{\alpha_{V}}) \rightarrow (V \times \C, \Omega_{V}) $ is a symplectomorpshim. 

Next, we discuss bundles over $\nu_{M}(H) \cong P \times_{\bar{\rho}} \D(\delta)$. 
Let $p_{\nu}: P \times \D(\delta) \rightarrow P \times_{\bar{\rho}} \D(\delta)$ be the natural projection, 
which is a principal $S^1$-bundle with the connection $1$-form 
$$
	\alpha_{\nu} =(1-r_1^2)\alpha+r_1^2 d\theta_{1}.
$$ 
Composing $\varphi_{\nu}$ with $p_{\nu}$, 
we obtain the principal $S^1$-bundle $\varphi_{\nu} \circ p_{\nu}: P \times \D(\delta) \rightarrow \nu_{M}(H)$ 
with the same connection $1$-form. 
Hence, the line bundle associated to $\varphi_{\nu} \circ p_{\nu}$ is given by 
$\Pi_{\nu}: (P \times \D(\delta)) \times_{\rho} \C \rightarrow \nu_{M}(H)$. 
Its total space 
is equipped with a symplectic structure $\omega_{\alpha_{\nu}}$ that lifts to the $2$-form on 
$(P \times \D(\delta)) \times \C $ defined by 
\begin{eqnarray}\label{eqn: symplectic form}
	d\alpha_{\nu} + d(r_{2}^2d\theta_{2}) + d(r_{2}^2 \alpha_{\nu}) = d((1+r_{2}^2)(\alpha_{\nu}+d\theta_{2})).
\end{eqnarray}

We claim that one can glue two line bundles $\Pi_{V}$ and $\Pi_{\nu}$ together symplectically, and 
the resulting bundle forms the line bundle $\Pi: L= (V\times \C) \cup ((P \times \D(\delta)) \times_{\rho} \C) \rightarrow M$ 
with $c_{1}(L)=- [\omega/2\pi]$.
To check this claim, 
enlarge the tubular neighborhood $\nu_{M}(H)$ slightly and take an extension 
$\hat{\varphi}_{\nu}: P \times_{\bar{\rho}} D^{2}(\delta') \rightarrow H$ for some $\delta'$ slightly greater than $\delta$.
%Set $A(\delta, \delta') = \{ z \in \C \mid \delta \leq |z| \leq \delta'\}$. 
We think of the image $\varphi_{\nu}(P \times_{\bar{\rho}} A(\delta, \delta')) \subset V$ as a collar neighborhood of $\del V$ in $V$ 
and set $\nu_{V}(\del V) = \varphi_{\nu}(P \times_{\bar{\rho}} A(\delta, \delta'))$.
Define a gluing map $\Phi: (P \times A(\delta, \delta')) \times_{\rho} \C  \rightarrow \nu_{V}(\del V) \times \C$ of the bundles by 
$$
	\Phi([x,(r_{1}, \theta_{1}), (r_{2}, \theta_{2})])=(\varphi_{\nu}([x,(r_{1},\theta_{1})]), (r_{2}, \theta_{1}+\theta_{2})). 
$$
It is easy to check that $\Phi^{*}d((1+r^2)(\lambda + d\theta)) = d(1+r_{2}^2)(\alpha_{\nu}+d\theta_{2})$ because 
${\varphi_{\nu}}^{*}\lambda = (1-r_{1}^{2})(\alpha-d\theta_{1})$, which implies 
that two bundles $\Pi_{V}$ and $\Pi_{\nu}$ are glued together symplectically. 
By construction, 
the curvature forms of the connections on $\Pi_{V}$ and $\Pi_{\nu}$ are $\omega|_{V}$ and $\omega|_{\nu_{M}(H)}$, respectively. 
Therefore, $c_{1}(L) = -[\omega/2\pi]$.

\section{Lefschetz-Bott fibrations}\label{section: Lefschetz-Bott}

\subsection{Fibered Dehn twists}\label{section: FDT}

%We briefly review fibered Dehn twists assuming that the reader is familiar with Dehn twists 
%(see \cite[(16.c)]{SeiBook} for the definition of a Dehn twist). 
%In the following, we define a submanifold of a symplectic manifold. 
%Although we deal with the case of codimension $c \geq 1$ in general in the appendix, 
%it suffices to consider only the case where $c=1$ to prove Theorem \ref{theorem}. 

%Let $p: P \rightarrow (B,\omega)$ be a principal $SO(c+1)$-bundle over a symplectic manifold $(B,\omega_{B})$ and 
%$\alpha$ a connection $1$-form of this bundle. 
%Let $T^{*}_{ \leq \lambda}S^{c}$ be the cotangent disk bundle of length $\lambda$. with the canonical symplectic structure 
%$\omega_{can}$. 
%We equip $(T^{*}_{\leq \lambda}S^{c}, \omega_{can})$ with the Hamiltonian $SO(c+1)$-action with 
%moment map $\Phi: T^{*}_{\leq \lambda}S^{c} \rightarrow \mathfrak{so}(n+1)^{*}$.

We first recall a Dehn twist in the lowest dimensional case. 
Let $T^{*}S^{1}=\R \times S^{1}$ be the cotangent bundle with the canonical symplectic structure $\omega_{can}=d(td\theta)$, 
where $S^1=\R/\Z$ and $(t,\theta) \in \R \times S^1$. 
Let $S_{0}$ be the zero-section of the bundle.
The moment map $\Phi: T^{*}S^1 \setminus S_{0}\rightarrow \R$, $\Phi(t,\theta)=|t|$ defines the Hamiltonian $S^1$-action 
on $T^{*}S^1 \setminus S_{0}$ given by 
$$
	\sigma_{\varphi}(t,\theta)=
	\begin{cases}
		(t,\theta+\varphi) & (t>0), \\
		(t,\theta-\varphi) & (t<0),
	\end{cases}
$$
for $\varphi \in S^1$ and $(t,\theta) \in T^{*}S^1$.
Take a cut-off function $g:\R \rightarrow \R$ with $\text{Supp}(g) \subset (-\infty, \epsilon)$ (for some $\epsilon>0$)
such that $g(s)=1/2$ for $s$ near $0$. 
Define the map $\tau: (T^{*}S^{1}, \omega_{can}) \rightarrow (T^{*}S^1, \omega_{can})$, called a 
\textit{model Dehn twist}, by 
$$
	\tau(t,\theta) = 
	\begin{cases}
		\sigma_{g(|t|)}(t,\theta) & \text{if} \ (t,\theta)\in T^{*}S^{1} \setminus S_{0}, \\ 
		(t, \theta+1/2) & \text{if} \ (t,\theta) \in S_{0}.
	\end{cases}
$$
By definition, it is compactly supported and a symplectic automorphism of $(T^{*}S^1, \omega_{can})$. 
Moreover, it is independent of the choice of $g$ up to isotopy. 

%Let $(M, \omega)$ be a symplectic manifold. 
%A submanifold $C$ of $(M, \omega)$ is called a \textit{coisotropic submanifold} if 
%$(T_{x}C)^{\omega} \subset T_{x}C$ at every point $x\in C$, where $(T_{x}C)^{\omega}$ is the symplectic complement of $T_{x}C$.
%A coisotropic submanifold $C \subset (M,\omega)$ is said to be \textit{spherically fibered} if 
%$C$ is a submanifold $C \subset M$ of codimension $c \geq 1$ such that: 
%the null foliation of $C$ is fibrating over a symplectic manifold $B$ with fiber a $c$-dimensional sphere $S^{c}$; 
%the structure group of $p:C \rightarrow B$ is equipped with a reduction to $SO(c+1)$, that is, 
%a principal $SO(c+1)$-bundle $P\rightarrow B$ and a bundle isomorphism $P \times_{SO(c+1)}S^{c} \cong C$. 
%For example, if $C$ is Lagrangian especially, 
%$C$ is isomorphic to the associated bundle $P \times_{SO(n+1)} S^{n}$ to a principal $SO(n+1)$-bundle $P$ over a point, and hence 
%$C$ is a so-called \textit{framed Lagrangian sphere} \cite[(16.a)]{SeiBook}. 

Let $(B,\omega_{B})$ be a symplectic manifold with $[\omega_{B}/2\pi] \in H^{2}(B;\Z)$ and 
$p: (P, \alpha) \rightarrow (B, \omega_{B})$ the Boothby-Wang bundle over $(B,\omega_{B})$. 
%Since $T^{*}S^1$ admits the $S^{1}$-action defined by $\varphi \cdot (t, \theta)= (t, \theta+\varphi)$, 
%The moment map $T^{*}S^1 \rightarrow \R$, $(t,\theta) \mapsto t$ defines (see \cite[Exercise 5.2.2]{MS3}) 
Considering the Hamiltonian $S^1$-action on $T^{*}S^1$ defined by 
$(t,\theta) \cdot \varphi = (t,\theta+\varphi)$ for $(t,\theta) \in T^{*}S^{1}$ and $\varphi \in S^{1}$, 
we can associate the $T^{*}S^1$-bundle $P \times_{S^1} T^{*}S^1 \rightarrow B$ to $p$, where $S^{1}$ acts 
on $P \times T^{*}S^1$ by 
$$
	 (x,(t,\theta)) \cdot \varphi=(x\cdot \varphi, (t,\theta-\varphi)).
$$
Equip the total space $P \times_{S^1} T^{*}S^1$ with the symplectic structure $\omega_{P}$ given by 
$$
	\omega_{P}={p}^{*}\omega_{B} + d(td\theta) + d(t\alpha). 
$$ 
Let us define a \textit{model fibered Dehn twist} $\tau_{P}$ by 
$$
	\tau_{P}([x,(t,\theta)])=[x,\tau(t,\theta)], 
$$
which is a compactly supported symplectic automorphism of $(P\times_{S^{1}} T^{*}S^{1}, \omega_{P})$. 

Let $(V, \omega)$ be a Liouville domain whose boundary is the Boothby-Wang bundle $p:(P, \alpha) \rightarrow (B,\omega_{B})$. 
We identify a collar neighborhood of $\del V$ with $((-\epsilon', 0] \times P, d(e^t \alpha))$.
Since $P$ is a codimension $1$ submanifold of $V$, $\{pt\} \times P$ is a coisotropic submanifold of $(V, \omega)$. 
By the tubular neighborhood theorem of coisotropic submanifolds \cite[p.124]{MS3}, 
a neighborhood $\nu$ of $\{-\epsilon'/2 \} \times P$ in $V$ is symplectomorphic to 
a neighborhood $\nu'$ of $P \times_{S^1} S_{0}$ in $P\times_{S^1}T^{*}S^{1}$. 
Denoting this symplectomorphism by $\phi: \nu \rightarrow \nu'$, we define the map $\tau_{\del}: V \rightarrow V$ by 
$$
	\tau_{\del}= 
	\begin{cases}
		\phi \circ \tau_{P} \circ \phi^{-1} & \text{on}\   \nu, \\
		\text{id} & \text{on}\  V \setminus \nu,
	\end{cases}
$$
where we choose $\tau_{P}$ whose support is contained in $\nu'$. 
The map $\tau_{\del}$ is a compactly supported symplectic automorphism of $(V,\omega)$ and called a \textit{fibered Dehn twist} along 
$\del V$.

\begin{remark}
In general, a fibered Dehn twist can be defined for a spherically fibered coisotropic submanifold of a symplectic manifold \cite{Per, WW}. 
%A submanifold $C$ of $(M, \omega)$ is called a \textit{coisotropic submanifold} if 
%$(T_{x}C)^{\omega} \subset T_{x}C$ at every point $x\in C$, where $(T_{x}C)^{\omega}$ is the symplectic complement of $T_{x}C$.
A coisotropic submanifold $C$ of a symplectic manifold $(M,\omega)$ is said to be \textit{spherically fibered} if 
$C$ is a submanifold $C \subset M$ of codimension $c \geq 1$ such that: 
the null foliation of $C$ is fibrating over a symplectic manifold $B$ with fiber a $c$-dimensional sphere $S^{c}$; 
the structure group of $p:C \rightarrow B$ is equipped with a reduction to $SO(c+1)$, that is, 
a principal $SO(c+1)$-bundle $P\rightarrow B$ and a bundle isomorphism $P \times_{SO(c+1)}S^{c} \cong C$.  
For example, the above Boothby-Wang bundle $P =\del V$ is a spherically fibered coisotropic submanifold of $(V, \omega)$.
%As an another example, if $C$ is Lagrangian especially, 
%$C$ is isomorphic to the associated bundle $P \times_{SO(n+1)} S^{n}$ to a principal $SO(n+1)$-bundle $P$ over a point, and hence 
%$C$ is a \textit{framed Lagrangian sphere} (see \cite[(16.a)]{SeiBook}). 
\end{remark}

\subsection{Preliminaries of symplectic Lefschetz-Bott fibrations}\label{section: def of LBF}
Let $M$ be a $2n$-dimensional smooth manifold equipped with a closed $2$-form $\Omega$ and an almost complex structure $J$. 

\begin{definition}[\cite{Per, WW}]
Let $N$ be an almost complex submanifold  of $(M,J)$. 
The $2$-form $\Omega$ is said to be \textit{normally K\"{a}hler} near $N$ if there exists 
a tubular neighborhood $\nu_{M}(N)$ of $N$ in $M$, foliated by normal slices $\{D_{x} \subset \nu_{M}(N)\}_{x \in N}$, such that 
$J|_{TD_{x}}$ is integrable and $\Omega|_{TD_{x}}$ is K\"{a}hler for each $x \in D_{x}$.
\end{definition}

%Let $D^2$ be a closed disk with the standard orientation.
%For a smooth map $\pi: X \rightarrow Y$ form a smooth manifold $X$ to another $Y$, 
%we denote by $\Crit(\pi)$ the set of critical points of $\pi$ and 
%set $\Critv(\pi)= \Critv(\pi)$. 

\begin{definition}[\cite{Per,WW}]\label{def: Lefschetz-Bott}
A \textit{symplectic Lefschetz-Bott fibration} 
$(E, \pi, \Omega, J, j)$ on $E$ 
consists of an even dimensional open manifold $E$ with a closed $2$-form $\Omega$ and 
a smooth surjective map $\pi:E \rightarrow \C$ with $\Crit(\pi) \subset \mathring{E}$ and $\Critv(\pi) \subset \C$, 
where $\Crit(\pi)$ and $\Critv(\pi)$ are the set of critical points of $\pi$ and critical values of $\pi$, respectively; 
an almost complex structure $J$ defined in a neighborhood $\mathcal{U} \subset E$ of $\Crit(\pi)$; 
a positively oriented complex structure $j$ defined in a neighborhood $\mathcal{V} \subset D^2$ of $\Critv(\pi)$ 
satisfying the following conditions: 
\begin{enumerate}
\renewcommand{\theenumi}{\roman{enumi}}
\renewcommand{\labelenumi}{(\theenumi)}

\item\label{condition: restriction}  
the restriction $\pi|_{\mathcal{U}}: \mathcal{U} \rightarrow \mathcal{V}$ is a $(J, j)$-map; 

\item\label{condition: submanifold}  
the critical point set $\Crit(\pi)$ is a smooth submanifold with finitely many connected components; 

\item\label{condition: Hessian} 
the holomorphic normal Hessian $D^2\pi_{x}|_{T^{\otimes2}D_{x}}$
%, which is locally written as the matrix $(\frac{\del^{2} \pi}{\del z_{i}\del z_{j}})$, 
is non-degenerate at each point $x\in \Crit(\pi)$, 
where $D_{x}$ is a normal slice of a tubular neighborhood of $\Crit(\pi)$;

 \item\label{condition: non-degeneracy}  
$\Omega|_{\mathcal{U}}$ is non-degenerate, compatible with $J$ and normally K\"{a}hler near $\Crit(\pi)$;

\item\label{condition: fibers} 
 $\Omega$ is non-degenerate on $\ker (D\pi) \subset TE$.

%\red{Can we rephrase this condition to the standard form $z \mapsto z_1^2+ \cdots z_c^2$? See 
%[Arnold etal. Theorem p.187]}
%for each fiber $N_{x}$ of the normal bundle $N \rightarrow \Crit(\pi)$; 

\end{enumerate}

If $E$ is a compact manifold with boundary and codimension $2$ corners, replace the range of $\pi$ by a closed $2$-disk $D^2$ with the standard orientation, and 
the tuple $(E, \pi, \Omega, J, j)$ satisfies two more conditions: 
\begin{enumerate}
\renewcommand{\theenumi}{\roman{enumi}}
\renewcommand{\labelenumi}{(\theenumi)}
\setcounter{enumi}{5} 
\item
 the boundary $\del E$ consists of the vertical boundary $\del_{v}E$ and 
the horizontal boundary $\del_{h}E$ meeting at the corner,  
where 
$$
	\del_{v}E=\pi^{-1}(\del D^2) \ \text{and} \ \del_{h}E=\cup_{y\in D^2}\del(\pi^{-1}(y)). 
$$
We require that if $x$ lies in $\del_{h}E$, then $(\ker (D\pi_{x}))^{\Omega} \subset T_{x}\del_{h}E$;

\item\label{condition: horizontal triviality} 
the map $\pi$ restricts to a fibration on $\del_{v}E$, and on a neighborhood of $\del_{h}E$, $\pi$ is equivalent to the projection 
$$
	 \nu_{F}(\del E_{z}) \times D^2 \rightarrow D^2, 
$$
with $\Omega$ identified with the split form $\omega_{f}+\pi^{*}\omega_{b}$ for some $K>0$, 
where $E_{z}$ is the regular fiber of $\pi$ over $z$, $\omega_{f}= \Omega|_{TE_{z}}$ and $\omega_{b}$ is some symplectic form on $D^2$. 

\end{enumerate}
\end{definition}

For the sake of brevity, we often denote a symplectic Lefschetz-Bott fibration $(E,\pi, \Omega, J, j)$ by 
$\pi: (E,\Omega) \rightarrow \C$ or $\pi: E \rightarrow \C$. 
If $\dim(\Crit(\pi))=0$, 
a symplectic Lefschetz-Bott fibration $\pi$ is a \textit{Lefschetz fibration}.

\begin{remark}
Occasionally we focus on the topology of the total space of a symplectic Lefschetz-Bott fibration. 
Then, we use an alternative notion to symplectic Lefschetz-Bott fibration: 
A tuple $(E,\pi, \Omega, J,j)$ is a \textit{smooth Lefschetz-Bott fibration} if $E, \pi, J, j$ are the same as in Definition 
\ref{def: Lefschetz-Bott}; $\Omega$ is a closed $2$-form defined in a tubular neighborhood $\mathcal{U}$ of $\Crit(\pi)$; 
the tuple satisfies the conditions from (\ref{condition: restriction}) to (\ref{condition: non-degeneracy}) in Definition 
\ref{def: Lefschetz-Bott}.
\end{remark}

%\begin{remark}
%As Seidel pointed out in his book \cite[Remark 15.2]{SeiBook}...
%\begin{enumerate}
%\renewcommand{\theenumi}{\roman{enumi}}
%\renewcommand{\labelenumi}{(\theenumi)}
%\setcounter{enumi}{7} 
%\end{enumerate}
%\end{remark}

%
%\red{
%ADD THE FOLLOWINGS:
%\begin{itemize}
%\item matching paths;
%\item vanishing thimbles;
%\item vanishing bundles;
%\item distinguished basis.
%\end{itemize}
%}

%\begin{remark} A \emph{smooth} Lefschetz-Bott fibration on a smooth $2n$-manifold $W$ with codimension $2$ corners,
%is a smooth map  $\pi : W  \to D^2$  which satisfies the first and
%the second conditions listed in the Definition~\ref{definition: exact symplectic Lefschetz} 
%except the K\"{a}hlerness condition.
%\end{remark}

Next, we briefly review materials related to symplectic Lefschetz-Bott fibrations following \cite[Section 2.1]{WW}. 
It suffices to think a symplectic Lefschetz-Bott fibration $\pi: (E,\Omega) \rightarrow D^2$ over a closed $2$-disk 
because given a symplectic Lefschetz-Bott fibration over $\C$, 
one can take a disk $D^2 \subset \C$ containing all its critical values.
Moreover, without loss of generality, we may assume that for each critical value, there exists only one connected component of 
$\Crit(\pi)$. 

Let us take a point $z_{0} \in \del D^2$ and let $\Critv(\pi)=\{z_{1}, \ldots, z_{k}\}$. 
Set $(E_{z},\Omega_{E_{z}})=(\pi^{-1}(z), \Omega|_{\pi^{-1}(z)})$.
For any critical value  $z \in \Critv(\pi)$,  
a \textit{vanishing path}
is an embedded path $\gamma : [0,1] \rightarrow D^2$ such that
$\gamma(0)=z_0$, $\gamma(1)=z$ and $\gamma^{-1}(\Critv(\pi))=\{1\}$. 
To each such a path and the connected component $B$ of $\Crit(\pi)$ in $E_{\gamma(1)}$,
one can associate a smooth submanifold, called the \textit{vanishing thimble} $T_{\gamma}$ for $\gamma$, defined by 
$$
	T_{\gamma} = 
	\left(\cup_{t\in[0,1)} \{ x \in E_{\gamma(t)} \mid \lim_{t_{0} \rightarrow 1}\Gamma_{\gamma|[t,t_{0}]}(x) \in B \} \right) \cup B, 
$$
where $\Gamma_{\gamma|[t,t_{0}]}:E_{\gamma(t)} \rightarrow E_{\gamma(t_{0})}$ 
is the parallel transport along the restricted curve $\gamma|: [t,t_{0}] \rightarrow D^2$ 
determined by the canonical symplectic connection on $\pi^{-1}(\gamma([t,t_{0}]))$. 
Each intersection $T_{\gamma} \cap E_{\gamma(t)}$ is a coisotropic submanifold 
in $(E_{\gamma(t)}, \Omega_{E_{\gamma(t)}})$, and 
$C_{\gamma}=T_{\gamma} \cap E_{z_{0}=\gamma(0)}$ is called the \textit{vanishing cycle} for $\gamma$.
In fact, this coisotropic submanifold is spherically fibered (\cite[Proposition 2.3]{WW}). 

A \textit{distinguished basis} of vanishing paths is an ordered collection of
vanishing paths $(\gamma_1, \ldots, \gamma_k)$ each of which starts at $z_{0}$ and ends at a critical value 
such that the following condition holds: 
\begin{itemize}
\item The path $\gamma_j$ intersects $\gamma_k$ only at $z_{0}$ for $j\neq k$;
\item Thinking of the tangent vectors $\gamma'_{1}(0), \ldots, \gamma'_{k}(0)$ as vectors in the upper half-plane, 
they have a natural clockwise ordering. 
\end{itemize}
Given a symplectic Lefschetz-Bott fibration $\pi: (E,\Omega) \rightarrow D^2$ and 
a distinguished basis $(\gamma_{1}, \ldots, \gamma_{k})$ of vanishing paths based at $z_{0}\in \del D^2$, 
we obtain a collection $(C_{\gamma_{1}}, \ldots, C_{\gamma_{k}})$ of spherically fibered coisotropic submanifolds in 
the fiber $(E_{z_{0}}, \Omega|_{E_{z_{0}}})$. 
Conversely, based on \cite[Theorem 2.13]{WW}, 
one can show that given a symplectic manifold $V$ and 
a collection $(C_{1}, \ldots, C_{k})$ of spherically fibered coisotropic submanifolds in $V$, 
there is a symplectic Lefschetz-Bott fibration $\pi:E \rightarrow D^2$ with fibers symplectomorphic to $V$ such that 
a collection of vanishing cycles with respect to some distinguished basis of vanishing paths coincides with the given collection 
(cf. \cite[Lemma 16.9]{SeiBook}).

\subsection{Fibrations on line bundles}\label{section: LBF over C}

Now we construct a symplectic Lefschetz-Bott fibration on the total space of the line bundle $\Pi: L \rightarrow M$ constructed in 
Section \ref{section: line bundles}. 
Let us denote $(P\times \D(\delta)) \times_{\rho} \C$ by $P(\D(\delta), \C)$.  
Define maps $\pi_{V}: V \times \C \rightarrow \C$ and 
$\pi_{\nu}: P(\D(\delta), \C) \rightarrow \C$
 by  
\begin{eqnarray}
 & \pi_{V}(x,r, \theta) = (r, \theta), \\
 & \pi_{\nu}([x, (r_{1}, \theta_{1}), (r_{2}, \theta_{2})]) = (\mu(r_{1})r_{2}, \theta_{1}+\theta_{2}). 
\end{eqnarray}
Here $\mu: \R \rightarrow \R$ is a smooth function such that $\mu(r)=r$ for $r \leq \epsilon$, 
$\mu(r)\equiv1$ for $r$ near $\delta$ and $\mu'(r) \geq 0$ at any point $r$, 
where $\epsilon>0$ is sufficiently small (Figure \ref{fig: rho}). 
Set $\pi = \pi_{V} \cup \pi_{\nu}: L \rightarrow \C$, namely 
$$
\pi(p)=
\begin{cases}
	\pi_{V}(p) & \text{if} \ p \in V \times \C, \\
	\pi_{\nu}(p) & \text{if} \ p \in P(\D(\delta), \C).
\end{cases} 
$$
By definition, 
$\Crit(\pi)$ coincides with $H_{0}=\{ [x, z_{1}, z_{2}] \in P(\D(\delta), \C) \mid z_{1}=z_{2}=0 \}$, 
which can be seen as 
the zero-section of the $(\D(\delta) \times \C)$-bundle $P(\D(\delta), \C) \rightarrow H$.

\begin{figure}[t]
\begin{tabular}{cc}
\begin{minipage}{0.45\hsize}
	\centering
	\begin{overpic}[width=150pt,clip]{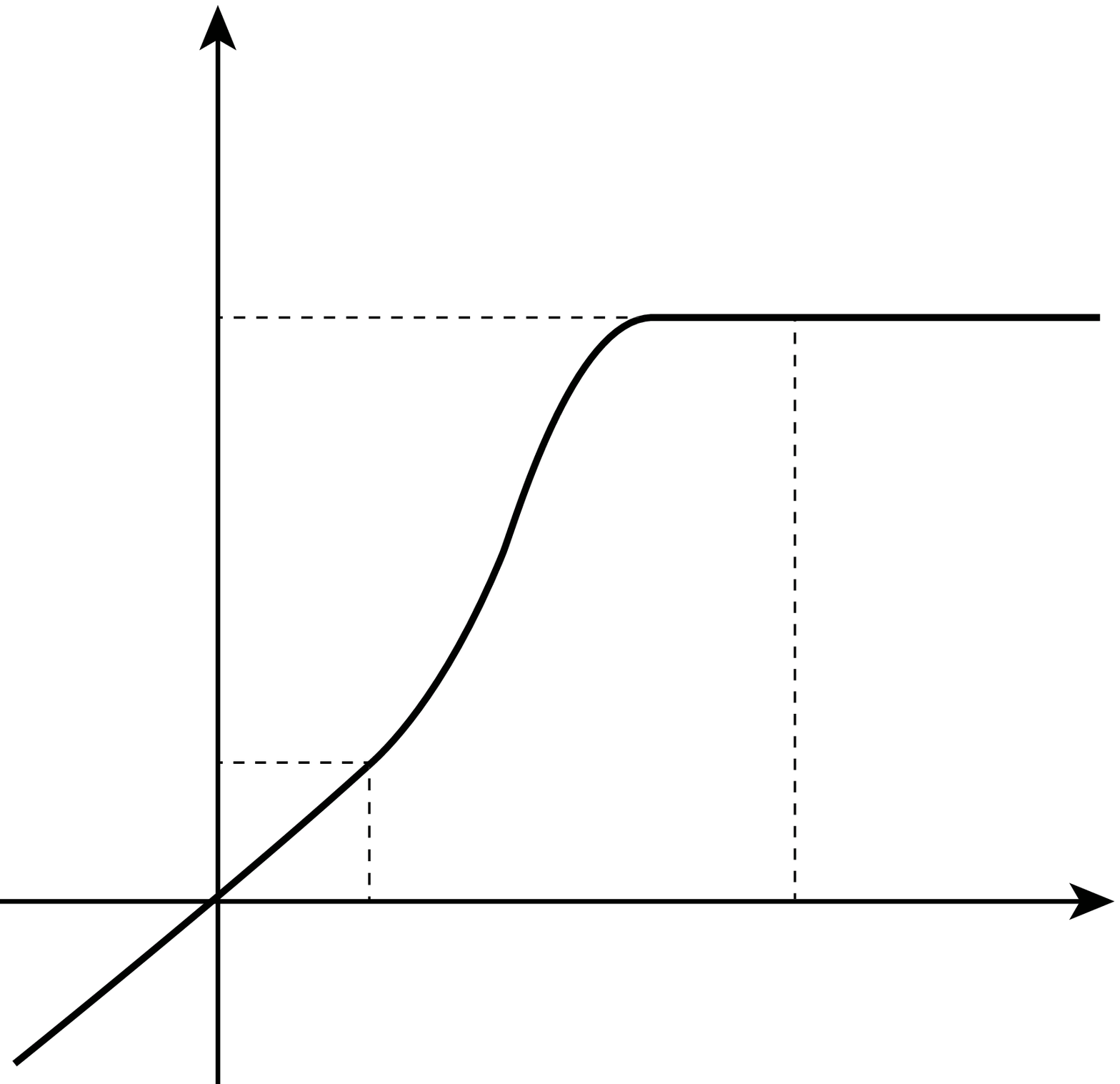}
	 \linethickness{3pt}
	\put(155,20){$r$} 
	\put(7,135){$\mu(r)$}
  	\put(31,15){$O$}
	\put(21,100){$1$}
	\put(103,15){$\delta$}
	\put(21,40){$\epsilon$}
	\put(48,15){$\epsilon$}
	\end{overpic}
	\caption{}
	\label{fig: rho}
\end{minipage}
\begin{minipage}{0.45\hsize}
%\end{figure}
%\begin{figure}[ht]
	\centering
	\begin{overpic}[width=150pt,clip]{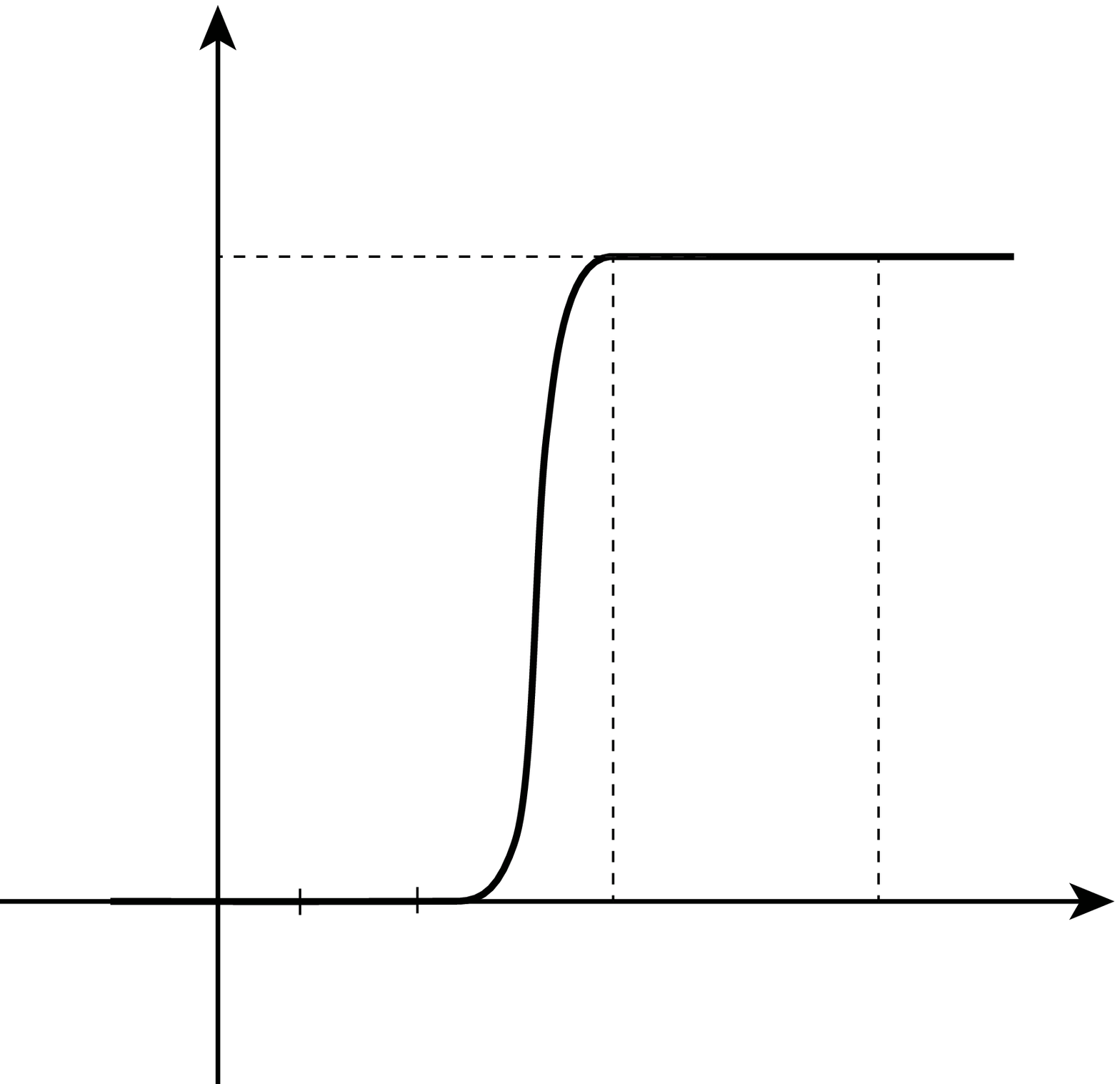}
	 \linethickness{3pt}
	\put(155,20){$s$} 
	\put(7,135){$u(s)$}
  	\put(19,15){$O$}
	\put(38,15){$\epsilon$}
	\put(52,15){$\epsilon'$}
	\put(78,15){$\epsilon''$}
	\put(116,15){$1$}
	\put(21,108){$1$}
	\end{overpic}
	\caption{}
	\label{fig: u}
\end{minipage}
\end{tabular}
\end{figure}

Recall how to define a symplectic form on $P(\D(\delta), \C)$. 
The $2$-form defined by (\ref{eqn: symplectic form}) descends to the symplectic form $\Omega_{0}$ on $P(\D(\delta), \C)$. 
Set 
$$
	\nu_{\epsilon}(H_{0}) = \{ [x, z_{1}, z_{2}] \in P(\D(\delta), \C) \, | \, |z_{1}|^2+|z_{2}|^2 \leq \epsilon \}
	\subset P(\D(\delta), \C).
$$
Regarding $\nu_{\epsilon}(H_{0})$ as the total space of the bundle $\nu_{\epsilon}(H_{0}) \rightarrow H$, 
we see that $\Omega_{0}$ does not agree with the standard symplectic form on each fiber of this bundle. 
To obtain a suitable symplectic form, we will deform the symplectic form $\Omega_{0}$.
Pick a smooth function $u:\R \rightarrow \R$ such that $u(s) \equiv 0$ for  $s \leq \epsilon'$, 
$u(s) \equiv 1$ for  $s \geq \epsilon''$ and 
$u'(s) \geq 0$ at any point $s$ for some small $\epsilon', \epsilon''>0$ with 
$\epsilon<\epsilon'<\epsilon''<1$ (Figure \ref{fig: u}). 
Set 
\begin{equation}\label{function}
	f(r_{1},r_{2})=u(r_{1}^2+r_{2}^2)
\end{equation}
for $(r_{1},r_{2}) \in [0,\delta) \times [0,\infty)$.
Define $\tilde{\Omega}_{1} \in \Omega^{2}(P\times \D(\delta) \times \C)$ by  
$$
	\tilde{\Omega}_{1}= d((1+r_{2}^{2})(d\theta_{2} + \alpha)) + d((1+f(r_{1}, r_{2})r_{2}^{2})r_{1}^{2}(d\theta_{1}-\alpha)). 
$$
It is easy to see that $\tilde{\Omega}_{1}$ descends to a $2$-form $\Omega_{1}$ on $P(\D(\delta), \C)$. 
Note that $\Omega_{1} = \Omega_{0}$ outside $\nu_{\epsilon''}(H_{0})$.

\begin{lemma}
Let $\Omega_{0}$ and $\Omega_{1}$ be $2$-forms on $P(\D(\delta), \C)$ defined above.  
Then, $\Omega_{1}$ is a symplectic form on $P(\D(\delta), \C)$. 
Moreover, there exists a symplectomorphism between $(P(\D(\delta), \C), \Omega_{0})$ and 
$(P(\D(\delta), \C), \Omega_{1})$ supported in $\nu_{\epsilon''}(H_{0})$. 
\end{lemma}

\begin{proof}
To show the first claim, 
it suffices to show the kernel of $(\tilde{\Omega}_{1})^{n+1}$ is spanned by $\del_{\theta_{1}}-\del_{\theta_{2}}+R_{\alpha}$. 
A direct computation  show that 
$$
	(\tilde{\Omega}_{1})^{n+1}= C(r_{1},r_{2}) 
	dr_{1}^{2} \wedge (d\theta_{1} - \alpha) \wedge dr_{2}^{2} \wedge (d\theta_{2} + \alpha) \wedge (d\alpha)^{n-1}, 
$$
where $C(r_{1},r_{2}) = n(n+1)(1-r_{1}^2+r_{2}^{2}(1-fr_{1}^2))^{n-1}(r_{1}^{2}r_{2}^{2}u'+fr_{2}^{2}+1) >0$.
This shows the claim. 

Since 
$$
	\Omega_{0} - \Omega_{1} = d((1-f)r_{1}^{2} r_{2}^{2}(d\theta_{1} - \alpha))=
	\begin{cases}
		d(r_{1}^{2} r_{2}^{2}(d\theta_{1} - \alpha)) & (r_{1}^2+r_{2}^2<\epsilon'), \\
		0 & (r_{1}^{2}+r_{2}^{2}>\epsilon''),
	\end{cases}
$$
the difference is an exact form and supported in $\nu_{\epsilon''}(H_{0})$. 
Thus, a standard Moser type argument shows the second claim (see \cite[Theorem 3.2.4]{MS3} for example). 
\end{proof}

\begin{lemma}\label{lemma: complex structure}
Let $\Omega_{1}$ be the symplectic form on $P(\D(\delta), \C)$ defined above. 
Then, there exists an almost complex structure $J$ on $P(\D(\delta), \C)$ compatible with $\Omega_{1}$ such that 
$H_{0}$ is an almost complex submanifold of $(\nu_{\epsilon}(H_{0}), J)$, and 
$\Omega_{1}$ is a normally K\"{a}hler near $H_{0}$.
\end{lemma}

\begin{proof}
We first construct a compatible almost complex structure defined in $\nu_{\epsilon}(H_{0})$. 
Let $J_{\alpha}$ be a compatible almost complex structure on $\ker(\alpha) \subset TP$ and 
$g_{\alpha}$ a Riemannian metric on $P$ given by 
$$
	(g_{\alpha})_{x}(u,v) =d\alpha_{x}(u,J_{\alpha}v)+\alpha_{x}(u)\alpha_{x}(v) 
$$
for $x\in P$ and $u,v \in T_{x}P$. 
Since $\alpha$ and $d\alpha$ are $S^1$-invariant, 
so is $g_{\alpha}$. 
Let $g_{1}$ and $g_{2}$ be the standard Riemannian metrics on $\D(\delta)$ and $\C$, respectively. 
The Riemannian metric $g_{\alpha}+g_{1}+g_{2}$ on $P \times \D(\delta) \times \C$ is $S^1$-invariant, 
and hence this induces a metric $g$ on the quotient space $P(\D(\delta), \C)$. 
The polar decomposition of $(P(\D(\delta), \C), \Omega_{1})$ with respect to the metric $g$ yields 
an almost complex structure $J$ on $\nu_{\epsilon}(H_{0})$. 
By construction, $J$ is compatible with $\Omega_{1}$, and $H_{0}$ is an almost complex submanifold of 
$(P(\D(\delta), \C), J)$.

Next, we see that $\Omega_{1}$ is normally K\"{a}hler near $H_{0}$. 
The tubular neighborhood $\nu_{\epsilon}(H_{0})$ is foliated by normal slices 
$ \{ D_{x} \}_{[x,0,0] \in H_{0}}$, where 
$$
	D_{x}= \{ [x,z_{1},z_{2}]\in P(\D(\delta), \C)\mid |z_{1}|^2+|z_{2}|^2 <\epsilon \}.
$$
By definition, the symplectic form $\Omega_{1}$ coincides with $r_{1}dr_{1} \wedge d\theta_{1} + r_{2}dr_{2} \wedge d\theta_{2}$ 
on each $D_{x}$, 
which is the standard symplectic form on the disk. 
Moreover, by construction the almost complex structure $J|_{TD_{x}}$ is the standard complex structure, in particular integrable and 
compatible with the symplectic form $r_{1}dr_{1} \wedge d\theta_{1} + r_{2}dr_{2} \wedge d\theta_{2}$. 
This completes the proof.
\end{proof}

Define a symplectic structure $\Omega$ on $L$ to be   
$$
	\Omega|_{V\times \C} = \Omega_{V} \ \text{and} \ \Omega|_{P(\D(\delta), \C)} = \Omega_{1}.
$$

\begin{proof}[Proof of Theorem \ref{theorem}]
We show that the tuple $(L, \pi, \Omega, J, j_{0})$ is a symplectic Lefschetz-Bott fibration. 
We choose $\nu_{\epsilon}(H_{0})$ and $\C$ as neighborhoods of $\Crit(\pi)$ and $\Critv(\pi)$, respectively. 
Since $\Crit(\pi) = H_{0}$, the tuple satisfies the condition (\ref{condition: submanifold}) of Definition \ref{def: Lefschetz-Bott}.
By the construction of $J$ in Lemma \ref{lemma: complex structure}, it is easy to see that 
the tuple meets the condition (\ref{condition: non-degeneracy}). 
Moreover, a straightforward computation shows that the tuple satisfies (\ref{condition: fibers}). 
Thus, it suffices to show the $(J,j_{0})$-holomorphicity of $\pi$ and the non-degeneracy of its holomorphic normal Hessian at each point of $H_{0}$.

On each normal slice $D_{x} \subset \nu_{\epsilon}(H_{0})$, we have $\mu(r_{1})=r_{1}$, and the map $\pi$ can be written as 
$$
	w=\pi([x,z_{1}, z_{2}])=z_{1}z_{2}
$$ 
by using complex coordinates $(z_{1},z_{2})$ of $D_{x}$ and $w$ of $\C$. 
Note that these are compatible with $J$ and $j_{0}$, respectively. 
Thus, $\pi|_{D_{x}}$ is $(J, j_{0})$-holomorphic, and in particular so is the whole $\pi$ 
because $J$ preserves $TD_{x}$ and its orthogonal complement with respect to the metric $g$ constructed above. 
Furthermore, 
the holomorphic normal Hessian of $\pi$ at $[x,0,0] \in H_{0}$ is given by 
$$
  \left(
    \begin{array}{cc}
       0 & 1 \\
      1 & 0
    \end{array}
  \right),
$$
which is non-degenerate. 
This finishes the proof. 
\end{proof}

\subsection{Fibrations with compact fibers}\label{section: compact fibers}
In the previous subsection, we saw that the line bundle $L$ admits a symplectic Lefschetz-Bott fibration $\pi:L\rightarrow \C$. 
Here, cutting the vertical and horizontal directions of the fibration, 
we show that a disk bundle associated to $L$ admits a symplectic Lefschetz-Bott fibration over the unit closed disk $\D$ 
with compact fibers.  

First, we construct the total space of our new fibration. 
One can define its vertical boundary by the preimage $\pi^{-1}(\del \D)$ of the boundary of $\D$. 
It suffices to define its horizontal boundary. 
To do this, 
we will find a function $h: P(\D(\delta), \C) \rightarrow \R$ satisfying that $dh(v)=0$ for any 
$v \in \ker (D_{[x,z_{1}, z_{2}]}\pi_{\nu})^{\Omega}$. 
Note that $V\times \C \subset L$ does not matter to the horizontal boundary. 
A straightforward computation shows that 
$$
\ker(D_{{[x,z_{1}, z_{2}]}}\pi_{\nu}) \cong T_{p_{H}(x)}H \oplus \R\langle \mu\del_{r_{1}}-\mu'r_{2}\del_{r_{2}}, 
\del_{\theta_{1}} - \del_{\theta_{2}} \rangle, 
$$
where $z_{j}=r_{j}e^{2\pi i \theta_{j}}$ and $p_{H}: (P, \alpha) \rightarrow (H, \omega_{H})$ is the Boothby-Wang bundle.
It turns out that 
its symplectic complement $(\ker(D_{{[x,z_{1},z_{2}]}}\pi_{\nu}))^{ \Omega}$ 
is spanned by 
$$
	2\mu'r_{2}^{2}\del_{\theta_{1}}+ \left( \mu\frac{\del 
	(r_{1}^{2}+ f r_{1}^{2}r_{2}^{2})}{\del r_{1}} -\mu'r_{2}\frac{\del(f r_{1}^{2} r_{2}^{2})}{\del r_{2}} \right)\del_{\theta_{2}},
$$
$$
	\frac{\del(r_{2}^{2}-f r_{1}^{2}r_{2}^{2})}{\del r_{2}} \del_{r_{1}}+
	\frac{\del(r_{1}^{2}+f r_{1}^{2}r_{2}^{2})}{\del r_{1}} \del_{r_{2}},  
$$
where $f$ is the function (\ref{function}). 
Now let us define $h: P(\D(\delta), \C) \rightarrow \R$ as a desired function by  
$$
	h([x,(r_{1},\theta_{1}), (r_{2}, \theta_{2})]) =
	-r_{1}^{2}+r_{2}^{2}-f(r_{1},r_{2}) r_{1}^{2}r_{2}^{2}. 
$$
It is actually a function of $r_{1}, r_{2}$ only, and 
we often write briefly $h(r_{1}, r_{2})$ for $h([x,(r_{1},\theta_{1}), (r_{2}, \theta_{2})])$. 
We can see that there is a unique point $(r_{1}, r_{2}) \in \R_{\geq 0} \times \R_{\geq 0}$ such that 
$h(r_{1}, r_{2})=c^2$ and $\rho(r_{1}) r_{2}=s$ for a sufficiently large $c>0$ and $s \leq 1$. 
Set 
$$
	P_{c}(\D(\delta), \C) =  \{ [x,(r_{1}, \theta_{1}), (r_{2}, \theta_{2})] \in P(\D(\delta), \C)\mid \rho(r_{1})r_{2} \leq 1, 
	h(r_{1}, r_{2}) \leq c^2 \},
$$
$$
	E_{c}= (V \times \D) \cup P_{c}(\D(\delta), \C).
$$
The glued manifold $E_{c}$ is diffeomorphic to a disk bundle over $M$ associated to $\Pi:L \rightarrow M$ by construction. 
We also note that the fibers of the restriction $\pi_{\nu}|_{P_{c}(\D(\delta), \C)}$ are compact, and in particular its regular fibers are diffeomorphic to $[0,1] \times P$. 
Thus, the map $\pi: L \rightarrow \C$ induces the desired map $\pi_{c}=\pi|_{E_{c}}: E_{c} \rightarrow \D$ with compact fibers.

\begin{figure}[t]
	\vspace{20pt}
	\centering
	\begin{overpic}[width=200pt,clip]{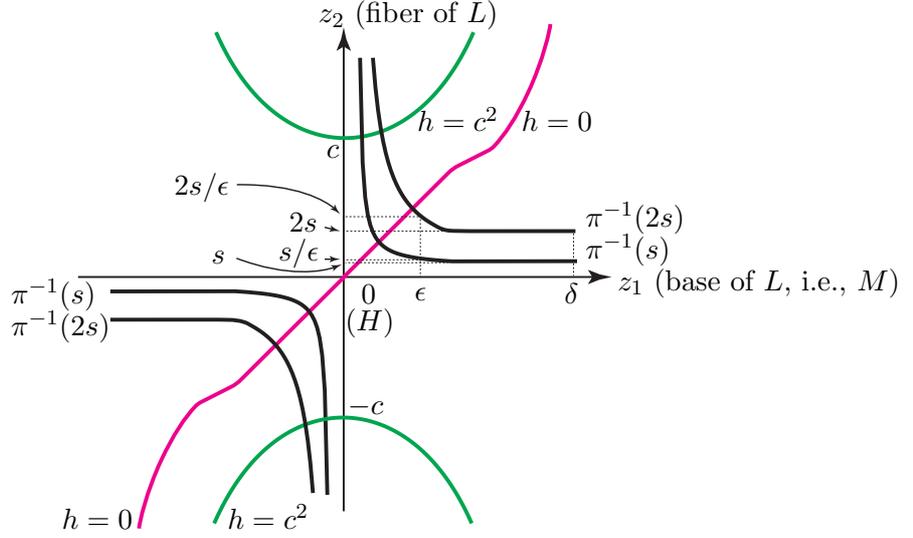}
	 \linethickness{3pt}
	\put(202,90){$z_{1}$ (base of $L$, i.e., $M$)} 
	\put(90,191){$z_{2}$ (fiber of $L$)}
  	\put(106,85){$0$}
	\put(100,75){($H$)}
	\put(126,86){$\epsilon$}
	\put(182,85){$\delta$}
	\put(-25,85){$\pi^{-1}(s)$}
	\put(-25,73){$\pi^{-1}(2s)$}
	\put(190,102){$\pi^{-1}(s)$}
	\put(190,114){$\pi^{-1}(2s)$}
	\put(127,150){$h=c^2$}
	\put(56,0){$h=c^2$}
	\put(93,140){$c$}
	\put(101,43){$-c$}
	\put(166,150){$h=0$}
	\put(-6,0){$h=0$}
	\put(79,112){$2s$}
	\put(75,101){$s/\epsilon$}
	\put(36,126){$2s/\epsilon$}
	\put(50,100){$s$}
	\end{overpic}
	\caption{Schematic picture of $\pi_{\nu,c}$. The thick black curves represent the fibers over $s$ and $2s$. 
	The red one represents $\Sigma$, and the green ones represent 
	part of the horizontal boundary of $E_{c}$.}
	\label{fig: graph of LB}
\end{figure}

This $\pi_{c}$ satisfies all conditions but (\ref{condition: horizontal triviality}) 
in Definition \ref{def: Lefschetz-Bott}. 
Thus, following the argument of \cite[Lemma 1.10]{Sei}, we next deform the symplectic structure $\Omega$ on $E_{c}$ so that 
it meets the remaining condition.
%analyzing the monodromy of $\pi_{c}$ along a loop around the origin. 
%Hereafter, we assume that $c$ is sufficiently large for simplicity. 
More precisely speaking, 
we will analyze the monodromy of $\pi_{c}$ along a loop in $\D$ by a direct computation and 
deform $\Omega$ so that the monodromy is compactly supported, which implies the horizontal triviality of $\pi_{c}$.

Since $\pi_{V}$ is a trivial fibration, we only need to consider $\pi_{\nu,c}:=\pi_{\nu}|_{E_{c}}$. 
Set $\Sigma = h^{-1}(0)$, 
which can be also written as 
$$
	\Sigma= H_{0} \cup (\cup_{z\in \D \setminus \{0\}} \{ p \in \pi_{c}^{-1}(z) \mid \lim_{t_{0}\rightarrow 1} \Gamma_{\gamma_{z}|_{[0,t_{0}]}}(p) \in H_{0} \}),
$$
where $H_{0}=\{ [x,z_{1},z_{2}] \in P_{c}(\D(\delta), \C)\mid z_{1}=z_{2}=0\} \cong H$ and 
$\Gamma_{\gamma_{z}|_{[0,t_{0}]}}$ denotes the parallel transport along the restriction of the path $\gamma_{z}:[0,1] \rightarrow \D$, 
$\gamma_{z}(t)=(1-t)z$ (cf. Figure \ref{fig: graph of LB}). 
Since $\Crit(\pi_{c})=H_{0}$ and $h$ is constant horizontally, 
one can trivialize $P_{c}(\D(\delta), \C) \setminus \Sigma$ symplectically via parallel transport in radial direction. 
Fix the identification $\psi:P \times_{S^1}(([-\delta^2, 0) \cup (0, c^2]) \times S^1) \rightarrow \pi^{-1}_{\nu,c}(0) \setminus \Sigma$ 
given by 
$$
	\psi:	[x,(t,\theta)] \mapsto 
	\begin{cases}	
		[x,0,(\sqrt{t}, -\theta)] & (t>0), \\
		[x, (\sqrt{-t}, \theta),0] & (t<0).
	\end{cases}
$$
%\begin{eqnarray*}
%	\psi_{+}: P \times_{\bar{\rho}} (\R_{>0} \times S^{1}) \rightarrow P_{+}, \ [x,t,\theta] \mapsto [x,0,(\sqrt{t}, -\theta)], \\ 
%	\psi_{-}: P \times_{\bar{\rho}} (\R_{<0} \times S^{1}) \rightarrow P_{-}, \ [x,t,\theta] \mapsto [x, (\sqrt{-t}, \theta),0], 
%\end{eqnarray*}
%where 
%$P_{+}= \{ [x,0, (r_{2}, \theta_{2})] \in P_c(\D(\delta), \C)\, | \, r_{2}>0 \}$ and 
% $P_{-}= \{ [x,(r_{1}, \theta_{1}), 0] \in P_c(\D(\delta), \C)\, | \, r_{1}>0 \}$. 
Note that $([-\delta^2, 0) \cup (0, c^2]) \times S^1$ is a subset of $T^{*}S^1 \setminus S_{0} \cong (\R \setminus \{ 0\}) \times S^{1}$.
We also remark that $\pi_{\nu,c}^{-1}(0) \setminus \Sigma$ consists of the following two parts: 
$$
	\{ [x,0, (r_{2}, \theta_{2})] \in P_c(\D(\delta), \C)\, | \, 0< r_{2}\leq c \},\ 
	\{ [x,(r_{1}, \theta_{1}), 0] \in P_c(\D(\delta), \C)\, | \, 0< r_{1}\leq \delta \}.
$$
Lifting the ray $\gamma_{\theta}(s) = se^{2\pi i \theta}$ with respect to the symplectic connection, 
we obtain the trivialization  $\Psi$ of ${P_{c}(\D(\delta), \C) \setminus \Sigma}$ as follows:  
\begin{eqnarray*}
	\Psi: (P \times_{S^{1}}(([-\delta^2, 0) \cup (0, c^2 ]) \times S^1)) \times \C \rightarrow P_{c}(\D(\delta), \C) \setminus \Sigma, \\
	\Psi([x,(t,\theta)], (s, \varphi)) = 
	\begin{cases} 
	[x, (r_{1}(s,t), \varphi+\theta), (r_{2}(s,t), -\theta)] & (t>0), \\ 
	[x, (r_{1}(s,t), \theta), (r_{2}(s,t), \varphi-\theta)] & (t<0),  
	\end{cases}
\end{eqnarray*}
where by definition $(r_{1}(s,t), r_{2}(s,t))$ is the solution of the following initial value problem derived from lifting 
$\gamma_{\theta}$: 
\begin{eqnarray}\label{ODE}
	& \frac{dr_{1}}{ds}=\frac{\frac{\del h}{\del r_{2}}}{r_{2}\frac{\del h}{\del r_{2}} -\mu(r_{1})\frac{\del h}{\del r_{1}}}, \ \ 
	 \frac{dr_{2}}{ds}=\frac{-\frac{\del h}{\del r_{1}}}{r_{2}\frac{\del h}{\del r_{2}} -\mu(r_{1})\frac{\del h}{\del r_{1}}}, \\
	& (r_{1}(0,t),r_{2}(0,t))= 
	 \begin{cases}
		(0,\sqrt{t}) & (t>0), \\
		(\sqrt{-t},0) & (t<0).
	\end{cases}
	\nonumber
\end{eqnarray}
If we restrict $\Psi$ to the fiber over a point $s>0$, it extends to a diffeomorphism 
$\psi_{s}$ between 
$P\times_{S^1} ([-\delta^2, c^2] \times S^1)$ and the whole fiber $\pi_{\nu,c}^{-1}(s)$.
Recall from the previous subsection that the symplectic structure $\Omega|_{P_{c}(\D(\delta), \C)}$ is given by 
$$ 
	d((1+r_{2}^{2})(d\theta_{2}+\alpha)+(r_{1}^{2}+f(r_{1}, r_{2})r_{1}^{2}r_{2}^{2})(d\theta_{1}-\alpha)).
$$	 
Let us denote this primitive $1$-form on $P_{c}(\D(\delta), \C) \setminus \Sigma$ by $\lambda_{\nu}$. 
Note that $\Omega$ is exact on $P_{c}(\D(\delta), \C) \setminus \Sigma$ but never on $P_{c}(\D(\delta), \C)$.
A straightforward computation yields  
\begin{eqnarray*}
	& \Psi^{*}(\lambda_{\nu})= 
	(1+|t|)(\alpha-d\theta)-\tilde{R}_{s}(t) \wedge d\theta, \\
	& \tilde{R}_{s}(t)=
	\begin{cases}
	-r_{2}^2(s,t)+t & (t>0), \\
	-r_{2}^{2}(s,t)-1 & (t<0). 
	\end{cases}
\end{eqnarray*}
Let $s_{0}>0$ be sufficiently small and 
$\ell_{s_{0}}:[0,1] \rightarrow \D$ the loop defined by $\ell_{s_{0}}(t)=s_{0}e^{2\pi it}$. 
The monodromy of $\pi_{c}$ along this loop is the conjugation $\psi_{s_{0}} \circ \tilde{\tau} \circ \psi_{s_{0}}^{-1}$ of the symplectic automorphism $\tilde{\tau}$ of 
$P \times_{S^1}(([-\delta^2, c^2\, ]) \times S^1)$ defined by 
$$
	\tilde{\tau}:[x,(t,\theta)] \mapsto [x,(t,\theta+\tilde{R}'_{s_{0}}(t))].
$$

Observe the behavior of $\tilde{R}_{s_{0}}'(t)$. 
If $\lim_{t \rightarrow \pm 0} \tilde{R}_{s_{0}}'(t)= \pm 1/2$ and $\tilde{R}_{s_{0}}'(t)=0$ for $t$ near $-\delta^2$ and $c^2$, 
then $\tilde{\tau}$ would be a fibered Dehn twist. 
%In fact, however, it is not a fibered Dehn twist. 
First, consider $\tilde{R}_{s_{0}}'(t)$ near $t=0$. 
Since $s_{0}$ is sufficiently small, 
we may assume by definition that $f(r_{1},r_{2})=0$, $h(r_{1},r_{2})= -r_{1}^{2}+r_{2}^{2}$ and $\mu(r_{1})=r_{1}$ for small $s$ and $|t|$.
This reduces the equations (\ref{ODE}) to 
$$
	\frac{dr_{1}}{ds}=\frac{r_{2}}{r_{1}^{2}+r_{2}^{2}}, \ 
	\frac{dr_{2}}{ds}=\frac{r_{1}}{r_{1}^{2}+r_{2}^{2}}.
$$
Solving the initial value problem, we have 
\begin{eqnarray*}
	(r_{1}(s,t), r_{2}(s,t))= 
	\begin{cases}
	\left( \frac{u(s,t)^{1/2}+u(s,t)^{-1/2}}{2} \sqrt{t},   \frac{u(s,t)^{1/2}-u(s,t)^{-1/2}}{2}\sqrt{t}\, \right)& (t>0), \\
	\left( \frac{u(s,t)^{1/2}-u(s,t)^{-1/2}}{2} \sqrt{-t}, \frac{u(s,t)^{1/2}+u(s,t)^{-1/2}}{2}\sqrt{-t}\, \right) & (t<0),
	\end{cases}
\end{eqnarray*}
where $u(s,t)=2s|t|^{-1}+\sqrt{4s^2 |t|^{-2}+1}$.
It can be seen directly that 
$$\lim_{t \rightarrow + 0} \tilde{R}_{s_{0}}'(t) = 1/2, \ \ \lim_{t \rightarrow - 0} \tilde{R}_{s_{0}}'(t)=-1/2.$$ 
Next, let us look at $\tilde{R}'_{s_{0}}(t)$ for $t$ near $-\delta^2$. 
Since $r_{1}(s_{0},t)$ is close to $\delta$ in this case, we may assume that $\mu(r_{1}(s_{0},t))\equiv 1$. 
We have $s_{0}=\mu(r_{1}(s_{0},t))r_{2}(s_{0},t)= r_{2}(s_{0},t)$, and 
it follows that $\tilde{R}'_{s_{0}}(t) = 0$ for $t$ near $-\delta^2$.
What is left is to see $\tilde{R}_{s_{0}}'(t)$ for $t$ near $c^2$.
As the constant $c$ is sufficiently large, we may assume that $f(r_{1},r_{2})=1$ and $h(r_{1},r_{2})=-r_{1}^{2}+r_{2}^{2}-r_{1}^{2}r_{2}^{2}$.
Thus, the initial value problem (\ref{ODE}) is reduced to 
$$
	\frac{dr_{1}}{ds}=\frac{r_{2}(1-r_{1}^{2})}{r_{1}^{2}+r_{2}^{2}}, \ 
	\frac{dr_{2}}{ds}=\frac{r_{1}(1+r_{2}^{2})}{r_{1}^{2}+r_{2}^{2}},\ (r_{1}(0,t),r_{2}(0,t))=(0, \sqrt{t}).
$$
A direct computation using Mathematica gives the solution
\begin{eqnarray*}
	  & r_{1}(s,t)  & = \frac{\sqrt{-|t|-s^2+\sqrt{t^2+4s^2+2|t|s^2+s^4}}}
	{\sqrt{2}}, \\
	 & r_{2}(s,t)  & = \frac{\sqrt{-|t|+s^2-\sqrt{t^2+4s^2+2|t|s^2+s^4}}}
	{\sqrt{-2-|t|-s^2+\sqrt{t^2+4s^2+2|t|s^2+s^4}}}. 
\end{eqnarray*}
We see that $\tilde{R}_{s}'(t) \rightarrow 0$ as $t \rightarrow +\infty$, but $\tilde{R}_{s}'(t) > 0$ for any $t$. 
Therefore, this implies that the monodromy of $\pi_{c}$ along $\gamma_{\theta}$ is not compactly supported. 
%For future use,
%we deform the symplectic structure $\Omega$ and make the monodromy compactly supported. 

To obtain a compactly supported monodromy, 
choose a cutoff function $\sigma \in C^{\infty}(\R_{>0}, \R)$ such that $\sigma'(t)\geq 0$ for every $t$, 
$\sigma(t)=0$ for small $t$ and $\sigma(t)=1$ for $t$ close to $c^2$. 
Take a $1$-form $\gamma$ on $P_{c}(\D(\delta), \C)$ such that 
$$
	\Phi^{*}\gamma = \sigma(t)\tilde{R}_{s}(t)d\theta.
$$
Set $\lambda_{\nu,c}= \lambda_{\nu}+\gamma$ and $\Omega_{c}= \Omega+d\gamma$. 
It is easy to check that 
$\lambda_{\nu,c}$ agrees with $\lambda_{\nu}$ near the boundary of each fiber of $\pi_{\nu,c}$, and 
$\Omega_{c}$ coincides with $\Omega$ on each fiber of $\pi_{c}$.  
The monodromy of $\pi_{\nu, c}$ along the loop $\ell_{s_{0}}$ is given by $\psi_{s_{0}} \circ \tau \circ \psi_{s_{0}}^{-1}$, where $\tau$ is 
the extension of the automorphism given by 
$$
	[x,(t,\theta)] \mapsto [x,(t,\theta+\{(1-\sigma(t))\tilde{R}_{s_{0}}(t)\}')].
$$
By the choice of $\sigma$, this is compactly supported and a fibered Dehn twist along the boundary of $\pi_{c}^{-1}(s_{0})$. 
The monodromy along a loop in $\D$ is well-defined up to symplectic isotopy independently of 
the choice of the representative of the homotopy class of the loop. 
Moreover, if one changes the base point of the loop, 
this contributes to only a conjugation of the initial monodromy. 

Since $\sigma(t)=0$ for small $t$ and $\Omega_{c}|_{\nu_{\epsilon}(H_{0})}=\Omega_{\nu_{\epsilon}(H_{0})}$, 
the almost complex structure $J$ on $\nu_{\epsilon}(H_{0})$ constructed in Lemma \ref{lemma: complex structure} and $j_{0}$ on $\D$ satisfy 
the conditions of a symplectic Lefschetz-Bott fibration. 
As for the triviality of the horizontal boundary, 
in the pull-back $2$-form 
$$
	 \Psi^{*}(\lambda_{\nu,c})= 
	(1+|t|)(\alpha-d\theta)- (1-\sigma(t))\tilde{R}_{s}(t) \wedge d\theta,
$$
the second term on the right-hand side vanishes close to $P \times_{S^1}(\{-\delta^2 ,c^2\} \times S^1)$, 
which concludes that $\Psi$ provides a trivialization of near $\del_{h}E_{c}$.
Therefore, 
$(E_{c}, \pi_{c}, \Omega_{c}, J, j_{0})$ is a symplectic Lefschetz-Bott fibration over $\D$.

We summarize the above discussion as follows.

\begin{proposition}
The tuple $(E_{c}, \pi_{c}, \Omega_{c}, J,j_{0})$ constructed above is a symplectic Lefschetz-Bott fibration over $\D$. 
Its monodromy along a small loop centered at the origin is a fibered Dehn twist along the boundary of the reference fiber. 
\end{proposition}

\subsection{Examples}\label{section: examples} 

%In this section, we give three examples related to Lefschetz-Bott fibrations. 
%In Section \ref{section: hypersurface} and \ref{section: cyclic}, we deal with K\"{a}hler manifolds. 
%In his paper \cite{Biran}, Biran introduced the notion of a \textit{polarized K\"{a}hler manifold}, 
%which is a tuple $\mathcal{P} = (M,\Omega,J; \Sigma)$ that consists of a K\"{a}hler manifold $(M,\omega,J)$ 
%with $[\Omega/2\pi]\in H^{2}(M;\Z)$ and  a smooth and reduced complex hypersurface $\Sigma \subset M$ 
%whose homology class $[\Sigma]$
% represents the Poincar\'{e} dual to $k[\Omega] \in H^2(M)$ for some $k \in \N$. 
% The number $k$ will be called the degree of the polarization $\mathcal{P}$. 
% In this paper, we only consider the case $k=1$.

\begin{example}[Hypersurface singularity]\label{example: hypersurface}

%\subsection{Hypersurface singularity}\label{section: hypersurface}

%In the following two subsections, we consider K\"{a}hler manifolds. 
%A \textit{smoothly polarized K\"{a}hler manifold} $\mathcal{P} = (M,\omega,J; H)$ is 
%a K\"{a}hler manifold $(M,\omega,J)$ 
%with $[\omega/2\pi]\in H^{2}(M;\Z)$ endowed with 
%a smooth and reduced complex hypersurface $H \subset M$ 
%whose homology class $[H]$
% represents the Poincar\'{e} dual to $k[\omega/2\pi] \in H^2(M; \Z)$ for some $k \in \N$ (see \cite{Biran}). 
% In this paper, we only consider the case $k=1$.

Consider the projective hypersurface $M_{d} \subset \CP^{n+1}$  degree $d$ defined by 
$$
	M_{d}=\left\{ [z_{0}:\ldots:z_{n+1}]\in \CP^{n+1} \;\middle|\; \sum_{j=0}^{n+1}z_{j}^d=0\right\}. 
$$
Let us denote by $\omega_{FS}$ the Fubini-Study form on $\CP^{n+1}$. 
Here, we take this form so that the cohomology class $[\omega_{FS}/2\pi]$ is Poincar\'{e} dual to 
the homology class of a hyperplane $h$ in $\CP^{n+1}$. 
The pull-back $\omega_{d}$ of $\omega_{FS}$ by the canonical inclusion map $i: M_{d} \hookrightarrow \CP^{n+1}$ is a symplectic 
form on $M_{d}$. 
%For any (real) $2$-dimensional submanifold $\Sigma$ in $M_{d}$, 
%\begin{eqnarray*}
%	\frac{1}{2\pi}\int_{\Sigma} {\omega_{d}} & = & \frac{1}{2\pi}\int_{\Sigma} {i^{*}\omega_{FS}}
%	 =  \frac{1}{2\pi}\int_{i(\Sigma)} {\omega_{FS}}
%	 =  \iota(\Sigma) \cdot h 
%	 =  \Sigma \cdot i^{-1}(h)
%	 =   \int_{\Sigma} {PD(i^{-1}(h))}
%\end{eqnarray*}
%(see \cite[p.59]{GH}).
Since the preimage $i^{-1}(h)$ can be seen as the hypersurface in $M_{d}$ given by 
$$ 
	H_{d} = \{ [z_{0}:\ldots : z_{n+1}] \in M_{d} \, | \, z_{n+1}=0 \}, 
$$
we have $[\omega_{d}/2\pi]=PD([H_{d}])$. 
Thus, by Theorem \ref{theorem}, 
a line bundle $L_{d} \rightarrow (M_{d}, \omega_{d})$ with $c_{1}(L_{d})=-[\omega_{d}/2\pi]$ 
admits a symplectic Lefschetz-Bott 
fibration $\pi: L_{d} \rightarrow \C$.

The total space $L_{d}$ of this line bundle has a singular theoretical meaning.
Consider the affine hypersurface in $\C^{n+2}$ given by 
$$ 
	X_{d} = \left\{ (x_{0}, \ldots, x_{n+1}) \in \C^{n+2} \;\middle|\; \sum_{j=0}^{n+1}x_{j}^d=0\right\},  
$$
which has a unique singularity at the origin.
%its link $P_{d}^{2n+1}=X_{d} \cap S^{2n+3}$ is the total space of the Boothby-Wang bundle over $(M_{d}, \omega_{d})$. 
%This is given as the restriction of the Hopf fibration $S^{2n+3} \rightarrow \CP^{n+1}$. 
%As explained in \cite[Example 3]{Muller}, we see that the associated line bundle to the 
%principal $S^1$-bundle 
%$P_{d}^{2n+1} \rightarrow M_{d}$ is a resolution of the singularity of $V_{d}^{n}$. 
Let $\widetilde{\C^{n+2}} = \{ (\ell, z) \in \CP^{n+1} \times \C^{n+2} \mid z \in \ell \}$, 
and 
$\sigma: \widetilde{\C^{n+2}}  \rightarrow \CP^{n+1}$, 
$\tau: \widetilde{\C^{n+2}}\rightarrow \C^{n+2}$ the projections to the first and second factors, 
respectively. 
Note that $\sigma$ is the tautological line bundle $\mathcal{O}(-1) \rightarrow \CP^{n+1}$, and $\tau$ is the blow-up of the origin. 
%\[
%   \xymatrix{
%    \widetilde{\C^{n+2}} \ar[r]^-{\tau} \ar[d]_{\sigma} & \C^{n+2} \\
%    \CP^{n+1},  & 
%   }
%\]
%where $\widetilde{\C^{n+2}} = \{ (\ell, z) \in \CP^{n+1} \times \C^{n+2} \, | \, z \in \ell \}$ and 
%$\mathcal{O}(-1)$ denotes the line bundle over $\CP^{n+1}$ of degree $-1$. 
The map $\sigma$ restricts to 
the line budle $\sigma|: \mathcal{O}(-1)|_{M_{d}} \rightarrow M_{d} \subset \CP^{n+1}$, and in fact  
$\mathcal{O}(-1)|_{M_{d}}$ is isomorphic to $L_{d}$. 
On the other hand, $\tau$ defines the restriction $\tau|: L_{d} \rightarrow X_{d}$: 
\[
   \xymatrix{
    \ar[r]^-{\tau|} L_{d} \cong \mathcal{O}(-1)|_{M_{d}}  \ar[d]_{\sigma|} & X_{d} \\
    M_{d} & 
   }
\]
It turns out that 
$\tau|$ contracts 
the zero-section of the bundle $\sigma|$ to the singular point of $X_{d}^{n+1}$. 
This shows that $L_{d}$ is a resolution of the singularity whose exceptional set is $M_{d}$.

Next we discuss a monodromy of a symplectic Lefschetz-Bott fibration $\pi_{c}: E_{d,c} \rightarrow \D$ with compact fibers 
obtained from $\pi: L_{d} \rightarrow \C$ as in Section \ref{section: compact fibers}. 
As we saw there, 
the monodromy of $\pi_{c}$ along $\del \D$ is a fibered Dehn twist along the boundary of the fiber $\pi_{c}^{-1}(1)$. 
We will see below that $\pi_{c}^{-1}(1)$ is diffeomorphic to the affine hypersurface $V_{d}(\delta_{0})$ given by 
$$ 
	V_{d}(\delta_{0})=\left\{ (x_{0}, \ldots, x_{n}) \in \C^{n+1} \;\middle|\; \sum_{j=0}^{n}x_{j}^{d}=1 \right\} \cap 
	\left\{ \sum_{j=0}^{n}|x_{j}|^{2} \leq \delta_{0}^2 \right\}.
$$ 

The following is known. 

\begin{theorem}[{\cite[Theorem 1.1]{AA}}]\label{thm: AA}
%Let $(V_{d}(\delta_{0}), \omega_{0})$ be the smooth manifold as above equipped with 
%the standard symplectic structure induced from the ambient space $(\C^{n}, \omega_{st})$. 
Let $\omega_{0}$ be a symplectic form on $V_{d}(\delta_{0})$ induced by the standard symplectic form on $\C^{n+1}$
Then, a fibered Dehn twist along the boundary $\del V_{d}(\delta_{0})$ is symplectically 
isotopic to a product of $d(d-1)^n$ Dehn twists in 
$\Symp_{c}(V_{d}(\delta_{0}), \omega_{0})$.
\end{theorem}

In our case, the symplectic structure on $V_{d}(\delta_{0})$ is different from the standard one $\omega_{0}$. 
We will show, however, that 
the above relation between a fibered Dehn twist and Dehn twists also holds under our symplectic structure. 

Let us specify our fiber $(F=\pi_{c}^{-1}(1), \omega_{F}=\Omega|_{F})$. 
In general, for a given closed symplectic manifold $(M, \omega)$ with $[\omega/2\pi]=PD[H]$, 
the fiber $(F,\omega_{F})$ is symplectomorphic to 
\begin{equation}\label{eqn: fiber}
	(V, 2\omega|_{V}=2d\lambda) \cup ([0,c_{0}] \times \del V, 2d(e^t\lambda|_{\del V})), 
\end{equation}
where $V=M \setminus \mathring{\nu}_{M}(H)$ and $c_{0}$ is some constant 
determined by the condition $h=c^2$ in Section \ref{section: compact fibers}. 
In our case, 
$(M,\omega)$ is $(M_{d}, \omega_{d})$, and $V$ is 
\begin{eqnarray*}
	& & (M_{d}\setminus \{ z_{n+1}=0\}) \cap \left\{ \sum_{j=0}^{n}|z_{j}/z_{n}|^2 \leq \delta_{0}^2 \right\} \cong  V_{d}^{n}(\delta_{0})
	%&\cong & \{ (x_{0}, \ldots, x_{n}) \in \C^{n+1} \mid \sum_{j=0}^{n}x_{j}^{d}=1, \sum_{j=0}^{n}|x_{j}|^{2} \leq \delta_{0}^2 \} \\
	%& =: &  V_{d}^{n}(\delta_{0})
\end{eqnarray*}
for some $\delta_{0}$. 
According to the discussion of \cite[Section 7] {Biran}, 
we may take 
$$\lambda_{d}=-d^{\C}\log\left(\sum_{j=0}^{n}|z_{j}|^{2}+1\right) = -d^{\C}\log(\|z\|^{2}+1)$$ on $V_{d}^{n}(\delta_{0})$ 
as the primitive $1$-form $\lambda$ in (\ref{eqn: fiber}), where $d^{\C}=d \circ J_{0}$. 
This $\lambda_{d}$ is different from the standard Liouville form on $V_{d}^{n}(\delta_{0})$. 
However, we conclude as follows. 

\begin{proposition}\label{proposition: relation}
Let $(F,\omega_{F})$ be the fiber $\pi_{c}^{-1}(1)$ as above. 
A fibered Dehn twist along the boundary $\del F$ is isotopic to 
a product of Dehn twists in $\Symp_{c}(F, \omega_{F})$.
\end{proposition}

To prove this, we need  the following lemma. 

\begin{lemma}\label{lemma}
Let $(V, \omega=d\lambda)$ be an exact symplectic manifold with convex boundary. 
For an arbitrarily positive number $a>0$, set $$(V_{a}, \omega_{a})=(V, \omega) \cup ([0,a] \times \del V, d(-e^{t}(\lambda|_{\del V})),$$
where $t \in [0,a]$. 
Then, $\Symp_{c}(V, \omega)$ is homotopy equivalent to $\Symp_{c}(V_{a}, \omega_{a})$.
\end{lemma}

\begin{proof}
The following argument is essentially the first half of the proof of  \cite[Proposition 2.1]{Evans}.
Let $\Phi^{t}$ be the flow of the Liouville vector field $X$ of $\lambda_{a}$ on $V_{a}$. 
Define a map $L_{a}: \Symp_{c}(V, d\lambda) \rightarrow \Symp_{c}(V_{a}, \omega_{a})$ by 
$
	L_{a}(\varphi)= \Phi^{-a} \circ \varphi \circ \Phi^{a}.
$
The inclusion $\iota_{a}: V \hookrightarrow V_{a}$ induces the inclusion 
$\iota_{a}^{S}: \Symp_{c}(V, d\lambda) \hookrightarrow \Symp_{c}(V_{a}, \omega_{a})$. 
The map $L_{a}$ is a homotopy inverse for $\iota_{a}^{S}$.
\end{proof}

\begin{proof}[Proof of Proposition \ref{proposition: relation}]
First, we prove that $\pi_{0}(\Symp_{c}(F, \omega_{F})) \cong \pi_{0}(\Symp_{c}(V_{d}^{n}(\delta_{0}), d\lambda_{st})).$ 
Set $\lambda_{\log}= -d^{\C}\log(\| z\|^2+1)$ and $\omega_{\log}=d\lambda_{\log}$. 
By Lemma \ref{lemma}, $\Symp_{c}(F, \omega_{F})$ is homotopy equivalent to 
$\Symp_{c}(V_{d}^{n}(\delta_{0}), \omega_{\log})$. 
Let $(\hat{V}_{d}^{n}(\delta_{0}), \hat{\omega}_{\log})$ be the symplectic completion of 
$(V_{d}^{n}(\delta_{0}), \omega_{\log})$. 
%, that is, 
%the manifold $\hat{V}_{d}^{n}(\delta_{0})=V_{d}^{n}(\delta_{0}) \cup [0, \infty) \times \del V_{d}^{n}(\delta_{0})$ 
%equipped with the symplectic structure $\hat{\omega}_{\log}$ that coincides with 
%$\omega_{\log}$ on $V_{d}^{n}(\delta)$ and 
%$d(e^t \lambda_{\log}|_{\del V_{d}^{n}(\delta_{0})})$ on $[0,\infty) \times \del V_{d}^{n}(\delta_{0})$.
By \cite[Lemma 2.2]{Evans}, 
this completion is symplectomorphic to $(X_{d}^{n}(1), \omega_{0})$, where 
$$
	X_{d}^{n}(1)= \left\{ (x_{0}, \ldots, x_{n}) \in \C^{n+1} \;\middle|\; \sum_{j=0}^{n}x_{j}^{d}=1 \right\}. 
$$
Moreover, according to \cite[Proposition 2.1]{Evans}, the inclusion maps 
$\Symp_{c}(V_{d}^{n}(\delta_{0}), \omega_{\log}) \hookrightarrow \Symp_{c}(\hat{V}_{d}^{n}(\delta_{0}), \hat{\omega}_{\log})$ and 
$\Symp_{c}(V_{d}^{n}(\delta_{0}), \omega_{0}) \hookrightarrow \Symp_{c}(X_{d}^{n}(1), \omega_{0})$ are 
weakly homotopy equivalences. 
Therefore, we conclude that 
\begin{eqnarray*}
	& & \pi_{0}(\Symp_{c}(F, \omega_{F}))  \cong  \pi_{0}(\Symp_{c}(V_{d}^{n}(\delta_{0}), \omega_{\log})) 
	\cong  \pi_{0}(\Symp_{c}(\hat{V}_{d}^{n}(\delta_{0}), \hat{\omega}_{\log})) \\
	 & \cong &  \pi_{0}(\Symp_{c}(X_{d}^{n}(1), \omega_{0})) 
	 \cong  \pi_{0}(\Symp_{c}(V_{d}^{n}(\delta_{0}), \omega_{0})). 
\end{eqnarray*}
It can be checked that each isomorphism above assigns a (fibered) Dehn twist to each (fibered) Dehn twist. 
Combining this and Theorem \ref{thm: AA} completes the proposition. 
\end{proof}
%
%Combining Theorem \ref{thm: AA} with Proposition \ref{proposition: relation}, 
%we have the following mapping class relation in $\pi_{0}(\Symp(F,\omega_{F}))$: 
%$$
%	[\tau_{\del}]=[\tau_{S_{1}} \circ \cdots \circ \tau_{S_{d(d-1)^n}}], 
%$$
%where $S_{j}$ are appropriate Lagrangian spheres in $(F,\omega_{F})$.

\end{example}

\begin{remark} 
One can interpret the above relation from the point of view of symplectic fillings of the link of the singularity 
$0 \in X_{d}^{n+1}$ as follows. 
As Acu and Avdeck pointed out in \cite{AA}, 
the product of Dehn twists comes from a Lefschetz fibration on a smoothing 
of the singularity, i.e., a Milnor fiber of the singularity. 
It serves as a Stein filling of the link of the singularity. 
The fibered Dehn twist is obtained as the monodromy of a Lefschetz-Bott 
fibration $\pi_{c}: E_{d,c} \rightarrow \D$ on a resolution of the singularity along $\del \D$. 
After smoothing corners of $E_{d,c}$ (see Appendix \ref{section: smoothing}), 
it serves as a strong symplectic filling of a contact manifold supported by the induced open book, 
in fact which is contactomorphic to the link of the singularity. 
Thus, the mapping class relation is related to two fillings of the link of the singularity. 
This interpretation will play an important role in Section \ref{section: A_k} below.
\end{remark}

%
%$$
%	\Symp_{c}(W, d\theta) := \{ \varphi \in \Symp(W, d\theta) \mid \text{Supp}(\varphi) \text{ is a compact set of }W\setminus \del W \}
%$$
%
%\begin{proposition}
%$\Symp_{c}(W, d\theta)$ is weakly homotopy equivalent to $\Symp_{c}(W, d\lambda_{st})$. 
%\end{proposition}
%
%\begin{proposition}
%The blow-up is the disk bundle associated to a Boothby-Wang bundle.
%\end{proposition}

\begin{example}[Cyclic quotient singularity]\label{example: cyclic}
%\subsection{Cyclic quotient singularity}\label{section: cyclic}

Consider $(\CP^{n}, m\omega_{FS})$ and its hypersurface $H_{m}^{n-1}$ given by 
$$
	\left\{ [z_{0}: \ldots: z_{n}]\in\CP^{n} \;\middle|\; \sum_{j=0}^{n} z_{j}^{m}=0  \right\}.
$$
Since $[H_{m}^{n-1}]=m[h]\in H_{n-2}(\CP^{n}; \Z)$ for a hyperplane $h \subset \CP^{n}$,  
we have $[m\omega_{FS}/2\pi]=PD[H_{m}^{n-1}]$. 
Therefore, by Theorem \ref{theorem}, a line bundle $L_{m} \rightarrow (\CP^{n}, m\omega_{FS})$ with 
$c_{1}(L_{m})=-[m\omega_{FS}/2\pi]$ admits 
a symplectic Lefschetz-Bott fibration. Note that this line bundle is isomorphic to 
the line bundle $\mathcal{O}(-m) \rightarrow \CP^{n}$ of degree $-m$, 
where 
$$
	\mathcal{O}(-m)=\mathcal{O}(-1)^{\otimes m}=
	\{ (\ell, x_{1}\otimes \cdots \otimes x_{m}) \in   \CP^{n}\times(\C^{n+1})^{\otimes m} \mid  x_{1},\ldots,x_{m} \in \ell \} .
$$ 

Similar to the previous example, 
this line bundle has a singular theoretical meaning. 
Let $\tau: \mathcal{O}(-1) \rightarrow \C^{n+1}$ and $\sigma: \mathcal{O}(-1) \rightarrow \CP^{n}$ be 
the projections as in the previous example. 
Let us denote by $\sigma_{m}: \mathcal{O}(-m) \rightarrow \CP^{n}$ the projection to the first factor. 
Considering the projection $\mathcal{O}(-m) \rightarrow (\C^{n+1})^{\otimes m}$ to the second factor, 
its image is identified with the quotient space $\C^{n+1}/ G_{m}$ derived from the diagonal action of the group 
$G_{m}= \{ \xi \in \C\mid \xi^{m}=1  \}$ on $\C^{n+1}$. 
Write $\tau_{m}:\mathcal{O}(-m) \rightarrow \C^{n+1}/G_{m}$ for this projection. 
We obtain the following commutative diagram: 
\[
   \xymatrix{
     & \ar[ld]_-{\sigma} \mathcal{O}(-1) \ar[r]^-{\tau} \ar[d]_{\tilde{q}} & \C^{n+1} \ar[d]^{q} \\
    \CP^{n} & \ar[l]_-{\sigma_{m}} L_{m} \cong \mathcal{O}(-m)  \ar[r]^-{\tau_{m}} & \C^{n+1}/G_{m}, 
   }
\]
where $\tilde{q}:\mathcal{O}(-1) \rightarrow \mathcal{O}(-m)$ and $q:\C^{n+1} \rightarrow \C^{n+1}/G_{m}$ are the natural projections. 
The quotient map $q$ yields a singularity at the origin $[0] \in \C^{n+1}/G_{m}$, which is a so-called 
\textit{cyclic quotient singularity}. 
Thus, the total space of our line bundle $L_{m} \cong \mathcal{O}(-m)$ can be regarded as a resolution of the singularity. 
\end{example}

\begin{remark}\label{remark: Stein}
The link of this singularity is a so-called \textit{lens space} of dimension $2n+1$. 
It is known that if $n\geq 2$, the singularity is not smoothable \cite{Sch}. 
Moreover, as discussed in \cite[Section 6.2]{EKP}, \cite[Corollary 4.8]{PPP} and \cite[Example 6.5]{BCS}, 
the link bounds no $(2n+2)$-dimensional manifold which is decomposed into handles of indices $ \leq n+1$. 
Thus, Milnor's topological characterization of Stein domains \cite{Milnor} shows that any contact structure on the link of dimension 
$2n+1\geq 5$ is never Stein fillable.
\end{remark}

%\begin{example}[Fillings of the link of the $A_{k}$-singularity]

%\section{Application}
\section{$A_{k}$-type singularities and Lefschetz-Bott fibrations}\label{section: A_k}

\subsection{Fillings of the link of the $A_{k}$-type singularity}\label{section: fillings}

Let $(M, \xi)$ be a closed contact manifold. 
A \textit{strong symplectic filling} of $(M, \xi)$ is a compact symplectic manifold $(W, \omega)$ with 
$\del W = M$ such that there exists a Liouville vector field $X$ defined near $\del W$, pointing outwards along $\del W$, 
and satisfying $\xi=\ker(i_{X}\omega|_{TM} )$ (as cooriented contact structure).
Consider the complex polynomial 
$$f(x_{0},\dots,x_{n+1})=x_{0}^{2}+ \cdots +x_{n}^{2}+x_{n+1}^{k+1} $$ 
and set 
$$V_{k}=\{ x\in \C^{n+2} \mid f(x)=0 \}, \  \Sigma_{k}=V_{k} \cap S^{2n+3}.$$
The variety $V_{k}$ has a unique singularity at the origin. 
This is called the $A_{k}$-type singularity, and $\Sigma_{k}$ is called the \textit{link} of the singularity.
The link inherits a canonical contact structure from the standard contact sphere $S^{2n+3}$. 
We often denote this contact manifold just by $\Sigma_{k}$ omitting the contact structure from the notation. 

In this subsection, 
we construct more than one strong symplectic filling of $\Sigma_{k}$ using symplectic Lefschetz-Bott fibrations over a disk. 
Such a fibration can be written topologically as the fiber sum of smooth 
Lefschetz-Bott fibrations with a single singular fiber. 
For given smooth Lefschetz-Bott fibrations $\pi_{j}: E_{j} \rightarrow D^2$ ($j=1,2$) 
with the fibers diffeomorphic to $F$, 
the \textit{fiber sum} $E_{1} \#_{f} E_{2}$ is defined as follows: 
Take $z_{j} \in \del D^{2}$ ($j=1,2$) and set $F_{j}=\pi_{j}^{-1}(z_{j})$.
Take tubular neighborhoods $\nu(F_{j})$
of these fibers in the vertical boundaries $\del_{v}E_{j}$ so that 
they are diffeomorphic to $[-\epsilon, \epsilon] \times F$ for some small $\epsilon>0$. 
Using a fiber-preserving diffeomorphism of $\nu(F_{j})$, 
we can glue $E_{j}$ together along $\nu(F_{j})$ 
and construct a new manifold $E_{1} \#_{f} E_{2}$, 
which admits a smooth Lefschetz-Bott fibration $\pi: E_{1} \#_{f} E_{2} \rightarrow D^{2} \natural D^{2} \cong D^2$, 
where $\natural$ denotes a boundary connected sum. 
The next lemma follows from this construction immediately. 

\begin{lemma}\label{lemma: fibersum}
Let $\pi_{j}: E_{j} \rightarrow D^2$ ($j=1,2$) be smooth (or symplectic) Lefschetz-Bott fibrations 
with the regular fibers diffeomorphic to $F$. 
Then, we have $\chi(E_{1} \#_{f} E_{2})=\chi(E_{1})+\chi(E_{2}) - \chi(F)$. 
\end{lemma}

\begin{proposition}\label{prop: fillings}
The contact manifold $\Sigma_{k}$ has at least $\lceil k/2 \rceil$ strong symplectic fillings up to homotopy, 
where $\lceil \, \cdot \, \rceil$ denotes the ceiling function.
\end{proposition}

\begin{proof}
Let $DT^{*}S^{n}$ be a disk cotangent bundle over the sphere $S^{n}$. 
This bundle can be considered as the complement of a neigborhood of an $(n-1)$-dimensional smooth quadric $Q^{n-1}$ in 
an $n$-dimensional smooth quadric $Q^{n} \subset \CP^{n+1}$. 
Hence we equip $DT^{*}S^{n}$ with a symplectic structure $\omega$ derived from restricting the 
Fubini-Study form on $Q^{n-1}$, 
where $\omega$ can be written explicitly as 
$$\omega=-dd^{\C}\log \left(\sum_{j=0}^{n}|z_{j}|^2+1 \right).$$
Let $\lambda= -d^{\C}\log(\sum_{j=0}^{n}|z_{j}|^2+1)$ and $S_{0}$ 
the zero-section of $DT^{*}S^{n}$, and 
consider the open book $(DT^{*}S^{n}, \lambda; \tau_{S_{0}}^{k+1})$. 
Then, according to \cite[Theorem 6.5]{CDvK}, 
the contact manifold 
$OB(DT^{*}S^{n}, \lambda; \tau_{S_{0}}^{k+1})$ is contactomorphic to 
the Boothby-Wang orbibundle over a symplectic orbifold. 
In fact, it is contactomorphic to the link $\Sigma_{k}$.

To construct strong symplectic fillings of $\Sigma_{k}$, 
for $\ell=1,\cdots, \lceil k/2 \rceil$ consider a symplectic Lefschetz-Bott fibration whose fibers are $(DT^{*}S^{n},\lambda)$ and 
whose collection of vanishing cycles are 
$$
( \underbrace{\del DT^{*}S^{n},\ldots,\del DT^{*}S^{n}}_{\ell}, 
\underbrace{S_{0},\ldots,S_{0}}_{k+1-2\ell} ).
$$
Let us denote by $X_{\ell}$ the symplectic manifold associated to the fibration. 
Since the fibered Dehn twist $\tau_{\del}$ along $\del DT^{*}S^{n}$ 
is symplectically isotopic to $\tau_{S_{0}}^{2}$, 
the monodromy of the open book induced by each fibration is symplectically isotopic to $\tau_{S_{0}}^{k+1}$. 
Therefore, $X_{\ell}$ is a strong symplectic filling of $\Sigma_{k}$. 

Now we shall compute the Euler characteristic $\chi(X_{\ell})$ to distinguish our fillings.
Let $E_{1}$ (resp. $E_{2}$) be the total space of a symplectic Lefschetz-Bott fibration whose fibers are 
$(DT^{*}S^{n}, \omega)$ and 
whose monodromy is $\tau_{\del}$ (resp. $\tau_{S_{0}}$).
Note that $E_{1}$ is diffeomorphic to a disk bundle over a quadric $Q^{n}$, 
and $E_{2}$ is diffeomorphic to a $(2n+2)$-dimensional disk. 
Moreover, it is known that 
$\chi(E_{2})= \chi(Q^n) = ((-1)^{n+1}-1)/2+n+1$ (see \cite[Exercise 5.3.7 (i)]{Dimca}).
Since $X_{\ell}$ can be regarded as the fiber sum of $\ell$ copies of $E_{1}$ and $k+1-2\ell$ copies of $E_{2}$, 
applying Lemma \ref{lemma: fibersum} repeatedly, 
we have 
\begin{eqnarray*}
	\chi(X_{\ell}) & = &\ell \cdot \chi(E_{1})+(k+1-2\ell) \cdot \chi(E_{2})-(k-\ell)\cdot \chi(DT^{*}S^{n}) \\ 
	& = & \{1-2\chi(Q^{n})+\chi(S^{n})\} \ell + 
	(k+1)\cdot\chi(Q^{n})-k \cdot \chi(S^{n})\\
	& = & \{3+2(-1)^n-2(n+1)\} \ell + 
	(k+1)\cdot\chi(Q^{n})-k \cdot \chi(S^{n}). 
	\end{eqnarray*}
Since $3+2(-1)^n-2(n+1)=-2n+1+2(-1)^n\neq 0$, $\chi(X_{\ell})$ is a linear function of $\ell$, which shows that $X_{\ell}$'s are mutually non-homotopic. 
\end{proof}

\begin{remark}\label{remark: ST}
In the above proof, 
$E_{1}$ is diffeomorphic to the total space of a disk bundle over $Q^{n}$, and $E_{2}\#_{f}E_{2}$ is 
diffeomorphic to 
the one of a disk cotangent bundle of $S^{n+1}$. 
In the light of Example \ref{example: hypersurface}, 
they are a resolution and a smoothing of the $A_{1}$-type singularity, respectively. 
Hence, we conclude that our strong symplectic fillings of $\Sigma_{k}$ are obtained by switching successively 
these two spaces originated from the singularity. 
Smith and Thomas introduced a symplectic model of this operation in \cite[Section 3]{ST}.
\end{remark}

\subsection{Resolution of the $A_{k}$-type singularity}

To interpret Proposition \ref{prop: fillings} from the singular theoretical view point in the next subsection, 
here we review a resolution of the $A_{k}$-type singularity. 
%Especially we will resolve this singularity by blowing up the singularity several times. 

Let $V_{k}$ be the variety defined by the polynomial $f$ as above. 
Set $X_{0}=\C^{n+1}$ and $V_{k}^{(0)}=V_{k}$. Consider the blow-up 
$\tilde{\tau}_{1}: X_{1}\cong \mathcal{O}(-1) \rightarrow X_{0}$ of $X_{0}$ at $0$ and let $V_{k}^{(1)}$ denote the proper transform of $V_{k}^{(0)}$. Here, 
$$
	X_{1} =\{(x,[z]) \in \C^{n+1} \times \CP^{n} \mid x_{i}z_{j}-x_{j}z_{i}=0 \ \text{for any}\ i, j \}.
$$
%Let us see $V_{k}^{(1)}$ before the exceptional set of $\tau_{1}= \tilde{\tau}_{1}|_{V_{k}^{(1)}}: V_{k}^{(1)} \rightarrow V_{k}^{(0)}$. 
%To do this, take coordinates of $X_{1}$ as follows. 
Letting $U_{1,j}=\{ (x,[z]) \in X_{1} \mid z_{j} \neq 0 \}$, we have $X_{1}= \cup_{j}U_{1,j}$.
On each $U_{1,j}$, we take the coordinate function $\varphi_{1,j}=(y_{0}, \ldots, y_{n})$ defined by 
$$
y_{i} = \begin{cases}
		z_{i}/z_{j} & (i \neq j), \\
		x_{j} & (i=j).
		\end{cases}
$$
Then, the preimage $\tilde{\tau}_{1}^{-1}(V_{k}^{(0)})$ is given by the following equations: 
\begin{eqnarray*}
	& y_{j}^{2}(y_{0}^{2}+ \cdots + y_{j-1}^{2} + 1+ y_{j+1}^{2}+ \cdots + y_{j}^{k-1}y_{n}^{k+1}) = 0 & \text{on}\ U_{1,j}\  (j\neq n), \\
	& y_{n}^{2}(y_{0}^{2}+ \cdots y_{n-1}^{2}+ y_{n}^{k-1}) = 0 & \text{on}\ U_{1,n}.
\end{eqnarray*} 
Since the points of $\varphi_{1,j}(U_{1,j}) \subset \C^{n+1}$ with $y_{j}=0$ 
correspond to points of the exceptional set $E_{1}=\{ (0,[z]) \in X_{1}\}$ of $\tilde{\tau}_{1}$, 
the proper transform $V_{k}^{(1)}$ is given by 
\begin{eqnarray}
	& y_{0}^{2}+ \cdots + y_{j-1}^{2} + 1+ y_{j+1}^{2}+ \cdots + y_{j}^{k-1}y_{n}^{k+1} = 0 & \text{on}\ U_{1,j}\  (j\neq n),\nonumber \\%\label{eqn1} \\
	& y_{0}^{2}+ \cdots + y_{n-1}^{2}+ y_{n}^{k-1} = 0 & \text{on}\ U_{1,n} \label{eqn2}, 
\end{eqnarray} 
which is equivalent to the set  
$$
	V_{k}^{(1)}=\{ (x,[z]) \in X_{1}\, | \, z_{0}^2 + \cdots + z_{n-1}^{2} + x_{n}^{k-1}z_{n}^{2}=0\}.
$$
If $k=1,2$, $V_{k}^{(1)}$ is non-singular. 
%In the case $k=1$, the exceptional set $V_{k}^{(1)} \cap E_{1}$ of $\tau_{1}$ is given by 
%$\{u_{0}^2 + \cdots + u_{n}^{2}=0 \} \subset \CP^{n}$, which is a smooth quadric $Q^{n-1}$ in $\CP^{n}$.  
%In the case $k=2$, the exceptional set is given by 
%$\{u_{0}^2 + \cdots + u_{n-1}^{2} =0 \} \subset \CP^{n}$, 
%which is the cone over a smooth quadric $Q^{n-2}$ in $\CP^{n-1}$ with vertex $p=(0: \ldots:0:1) \in \CP^{n}$, 
%denoted by $\overline{Q^{n-2},p}$.
%In the case $k>2$, the exceptional set is the same as the case $k=2$. 
In the case $k>2$, it has a unique singularity at $p_{1}=(0,[0:\ldots:0:1]) \in V_{k}^{(1)}$. 
On $U_{1,n}$, the defining equation of $V_{k}^{(1)}$ is given by the equation (\ref{eqn2}), and 
it turns out that the singularity is nothing but of type $A_{k-2}$.  
Thus, we can resolve the singularity by taking the blow-up $\tau_{2}: V_{k}^{(2)} \rightarrow V_{k}^{(1)}$ of $V_{k}^{(1)}$ at the singularity as we did above. 
Thus, by induction on $k$, we conclude that the procedure of $ \lceil k/2 \rceil$ times successive blowing up resolves the initial $A_{k}$-type singularity. 
where $\lceil \, \cdot \, \rceil$ denotes the ceiling function.

\subsection{Lefschetz-Bott fibrations and blowing up}

Let $V_{k}$ be the variety $\{ f=0 \}$ as above and 
$\tau: \tilde{V}_{k} \rightarrow V_{k}$ the restriction of the blow-up at $0\in \C^{n+2}$. 
Although the resulting space $\tilde{V}_{k}$ still has a singularity in general, 
we can smooth it and denote its smoothing by $\tilde{V}_{k, \epsilon}$. 
As a model of this smoothing, we take the union of the following two spaces: 
$$
	U := \{ ((x_{0}, \ldots, x_{n},0), [z_{0}:\ldots:z_{n}:0]) \in \tilde{V}_{k}\mid z_{0}^2 +\cdots + z_{n}^{2}=0 \} \subset \C^{n+2} \times \CP^{n+1},
$$
$$
	W_{\epsilon} = \varphi_{1,n+1}^{-1} (\{(y_{0}, \ldots,y_{n+1}) \in \C^{n+2} \mid 
	y_{0}^{2}+\cdots+y_{n-1}^{2}+y_{n+1}^{k-1}= \epsilon \rho(\|y\|^2) \}), 
$$
where $\rho:\R_{\geq 0} \rightarrow \R$ is a smooth, monotonically increasing function such that $\rho(t) \equiv 1$ for $t< \epsilon$ and 
$\rho(t) \equiv 0$ for $t>1$. 
Note that $U$ is the total space of a line bundle over $Q^{n-1}$ and in particular smooth. 
According to \cite[Proposition 2.3]{Fauck}, $W_{\epsilon}$ is smooth, and so is $\tilde{V}_{k,\epsilon}$.
It is well known that 
$V_{k,\epsilon}=\{x \in \C^{n+2}\mid f(x)=\epsilon\}$ admits a Lefschetz fibration 
$\pi: V_{k,\epsilon} \rightarrow \C$ defined by the projection to the last coordinate: 
$$
	\pi(x_{0},\ldots, x_{n+1})=x_{n+1}. 
$$
The collection of its vanishing cycles is given by $(k+1)$-tuple $(S_{0}, \ldots, S_{0})$ of the zero-section 
of the cotangent bundle of $S^n$ with respect to some distinguished basis.
To state the following lemma and proposition, we recall one notion. 
Let $\pi_{i}: E_{i} \rightarrow \C$ ($i=1,2$) be smooth Lefschetz-Bott fibrations and 
$(C_{i,1}, \ldots, C_{i,k_{i}})$ ($i=1,2$) the ordered collections of vanishing cycles of $\pi_{i}$ for some distinguished basis. 
The two fibrations are said to be \textit{related by a monodromy substitution} if 
there exists an integer $j_{0}$ with $1 \leq j_{0} \leq \min\{ k_{1}, k_{2}\}$ such that 
$C_{1,j} = C_{2,j}$ for any $j<  j_{0}$, and
$\tau_{C_{1,j_{0}}} \circ \tau_{C_{1,j_{0}+1}}  \cdots \circ \tau_{C_{1,k_{1}}}$ and 
$\tau_{C_{2,j_{0}}} \circ \tau_{C_{2,j_{0}+1}} \circ \cdots \circ \tau_{C_{2,k_{2}}}$ are smoothly isotopic. 
Note that the products $\tau_{C_{1,1}} \circ \tau_{C_{1,2}} \circ  \cdots \circ \tau_{C_{1,k_{1}}}$ and 
$\tau_{C_{2,1}} \circ \tau_{C_{2,2}} \circ \cdots \circ \tau_{C_{2,k_{2}}}$ are smoothly isotopic.

\begin{lemma}\label{lemma: blow up}
The smoothing $\tilde{V}_{k,\epsilon}$ admits a smooth Lefschetz-Bott fibration that 
is related to $\pi$ by a monodromy substitution.
\end{lemma}

\begin{proof}
Define a map $\tilde{\pi}:\tilde{V}_{k,\epsilon} \rightarrow \C$ by 
$$
	\tilde{\pi}((x_{0},\ldots,x_{n+1}),[z_{0}:\ldots:z_{n+1}])=x_{n+1}.
$$
On $W_{\epsilon}$, the map $\tilde{\pi}$ restricts to a Lefschetz fibration. 
To see this, 
consider $$\tilde{\pi} \circ \varphi_{1,n+1}^{-1}: \varphi_{1,n+1}(W_{\epsilon}) \rightarrow \C, \ (y_{0},\cdots, y_{n+1}) \mapsto y_{n+1}.$$  
By the definition of $W_{\epsilon}$, we see that this map is a smooth Lefschetz fibration, 
which is similar to $\pi$. 

%We claim that $\tilde{\pi}|_{U}$ is a smooth Lefschetz-Bott fibration. 
A straightforward computation using local coordinates shows that on a neighborhood of $U$, $\tilde{\pi}$ is a smooth Lefschetz-Bott fibration whose 
critical point set is $\{x=0\} \subset U$, isomorphic to a quadric. 
It turns out that the monodromy sequence of $\tilde{\pi}$ with respect to some distinguished basis is given by 
$$
	(\tau_{S_{0}}, \ldots, \tau_{S_{0}}, \tau_{\del}) 
$$
that consists of $k-1$ Dehn twists along the zero-section of $T^{*}S^{n}$ and a fibered Dehn twist. 
The automorphism $\tau_{\del}$ is symplectically, in particular smoothly isotopic to $\tau_{S_{0}} \circ \tau_{S_{0}}$, 
which completes the proof.
\end{proof}

%\begin{remark}
%The fibration $\tilde{\pi}$ given in the proof of the theorem is motivated by the projection $V_{k} \rightarrow \C$, $x\mapsto x_{n+1}$.
%\end{remark}

\begin{proposition}\label{prop: resolution}
A resolution of the $A_{k}$-type singularity admits a smooth Lefschetz-Bott fibration which are 
related to the Lefschetz fibration $\pi$ on a Milnor fiber of the $A_{k}$-type singularity by a sequence of monodromy substitutions. 
\end{proposition}

\begin{proof}
As we saw in the previous subsection, 
a resolution of the $A_{k}$-type singularity is obtained by a sequence of successive blow-ups: 
$$
	V_{k}^{(\lceil k/2 \rceil)} \xrightarrow{\tau_{\lceil k/2 \rceil}} V_{k}^{(\lceil k/2 \rceil-1)} 
	\xrightarrow{\tau_{\lceil k/2 \rceil -1}} \cdots \xrightarrow{\tau_{1}} V_{k}^{(0)}=V_{k}.
$$
Here, 
each $\tau_{j}$ is the blow-up at the $A_{k-2(j-1)}$-type singularity in $V_{k}^{(j-1)}$. 
Applying Lemma \ref{lemma: blow up} to each blow-up $\tau_{j}: V_{k}^{(j)} \rightarrow V_{k}^{(j-1)}$, 
we show that the smoothing $V_{k,\epsilon}^{(j)}$ admits a smooth 
Lefschetz-Bott fibration that is related to one on $V_{k,\epsilon}^{(j-1)}$. 
Therefore, inductively we conclude the proposition. 
\end{proof}

%The following result is well-known for Lefschetz fibrations. 
%In fact, it also holds for Lefschetz-Bott fibrations. 

%\begin{proposition}
%Let $\pi: E \rightarrow D^2$ be a Lefschetz-Bott fibration with two critical values. 
%suppose two vanishing cycles  $C_{1}, C_{2}$ with respect to a distinguished basis are disjoint. 
%Then, the total space $E$ is diffeomorphic to the one of a Lefschetz-Bott fibration 
%with a unique singular fiber which contains critical sets corresponding to $C_{1}, C_{2}$.
%\end{proposition}

\appendix
\section{Lefschetz-Bott fibrations and fillings of contact manifolds}\label{section: smoothing}

In this appendix, we show that 
the total space of a symplectic Lefschetz-Bott fibration over the unit disk $\D$ serves as a strong symplectic filling of a contact manifold 
supported by an open book induced by the fibration. 
To prove this, we mainly follow a recent paper of Lisi, Van Horn-Morris and Wendl \cite{LVHMW} where they discussed a relation between 
$4$-dimensional Lefschetz fibrations and various symplectic fillings of contact $3$-manifolds. 

\subsection{Open books}\label{section: open book}

Let $V$ be a $2n$-dimensional compact smooth manifold with boundary equipped with an 
exact symplectic form $\omega=d\lambda$ such that 
the Liouville vector field of $\lambda$ is transverse to $\del V$ and points outwards along $\del V$. 
Then, the symplectic manifold $(V, d\lambda)$ is called a \textit{Liouville domain.}

\begin{definition}
An \textit{abstract open book} $(V,\lambda; \phi)$ consists of a Liouville domain $(V, d\lambda)$ and 
its compactly supported symplectomorphism $\phi \in \Symp_{c}(V, d\lambda)$. 
\end{definition}

Let us review that an abstract open book $(V, \lambda; \phi)$ provides a contact manifold. 
According to \cite[Section 7.3]{Geiges}, %(see also \cite[Section 2.11]{vanKoert}), 
the symplectomorphism $\phi$ is assumed to be compactly supported and exact, i.e., 
$\phi^{*}\lambda - \lambda=dH$ for some positive function $H$ (\cite[Lemma 7.3.4]{Geiges}). 
Note that since $\phi$ is compactly supported, $H$ is constant near $\del V$. 
We assume $H=2\pi$ near $\del V$.
Define the mapping cylinder $V_{H}(\phi)$ by 
$$
	V_{H}(\phi) = (V \times \R) / (\phi(x), \varphi+H(x)) \sim (x, \varphi). 
$$
Let $D^{2}(a)=\{ z \in \C \mid |z|\leq a\}$ and $A(a,b)=\{ z \in \C \mid a \leq |z|\leq b\}$. 
Identifying a collar neighborhood of $\del V_{H}(\phi)$ with $[-\epsilon, 0] \times \del V \times S^1$, 
we define a manifold by 
$$
	M(V, \lambda; \phi) = V_{H}(\phi) \cup_{\Phi} (\del V \times D^2(1+\epsilon)), 
$$
where $\Phi: \del V \times A(1,1+\epsilon) \rightarrow [-\epsilon, 0] \times \del V \times S^{1} \subset V_{H}(\phi)$ is 
the diffeomorphism given by 
\begin{eqnarray}\label{eqn: glue}
	\Phi: (x,re^{2\pi i \varphi}) \mapsto (1-r, x, \varphi).
\end{eqnarray}
The manifold $M(V, \lambda; \phi)$ is what we shall equip with a contact structure. 
The $1$-form $\alpha=\lambda+Kd\varphi$ for an arbitrary constant $K>0$ defines a contact structure on $V_{H}(\phi)$. 
On the collar neighborhood $[-\epsilon, 0] \times \del V \times S^1$ of $\del V_{H}(\phi)$, 
 $\alpha$ can be written as $e^{r}\lambda_{\del V}+Kd\varphi$. 
Here, $\lambda_{\del V}$ is the pull-back of $\lambda$ by the inclusion $\del V \hookrightarrow V$.
What is left is to extend $\alpha$ to $\del V \times D^2(1+\epsilon)$. 
Let us define a $1$-form on $\del V \times D^2(1+\epsilon)$ by 
$$
	\beta=f(r)\lambda_{\del V}+g(r)d\varphi,
$$ 
where $f(r)$, $g(r):[0,1+\epsilon] \rightarrow \R$ are smooth functions specified later.
Simple computations give 
${\Phi}^{*}(\alpha)=e^{1-r}\lambda_{\del V}+Kd\varphi$ and 
$$
	\beta \wedge (d\beta)^{n} = 
	nf(r)^{n-1}(f(r)g'(r)-f'(r)g(r)) \lambda_{\del V}\wedge (d\lambda_{\del V})^{n-1} \wedge dr \wedge d\varphi. 
$$
Thus, for $\beta$ being a contact form, the functions $f$ and $g$ satisfy the following conditions (see Figure \ref{fig: graph contact}): 
\begin{itemize}
\item $f(r)=C_{0}$ and $g(r)=C_{1}r^2$ for $r \leq \epsilon/2$, where $C_{0}>1$ and $C_{1}>0$ are constants; 
\item $f(r)=e^{1-r}$ and $g(r)=K$ for $r \geq 1$;
\item $(f(r), g(r))$ is never parallel to its tangent vector $(f'(r), g'(r))$ for $r \neq 0$.
\end{itemize}
Taking such functions $f$ and $g$, 
we obtain a contact form and hence a contact structure on $M(V, \lambda; \varphi)$. 
We denote the resulting contact manifold by $OB(V,\lambda; \phi)$. 

\begin{figure}[bht]
	\vspace{20pt}
	\centering
	\begin{overpic}[width=350pt,clip]{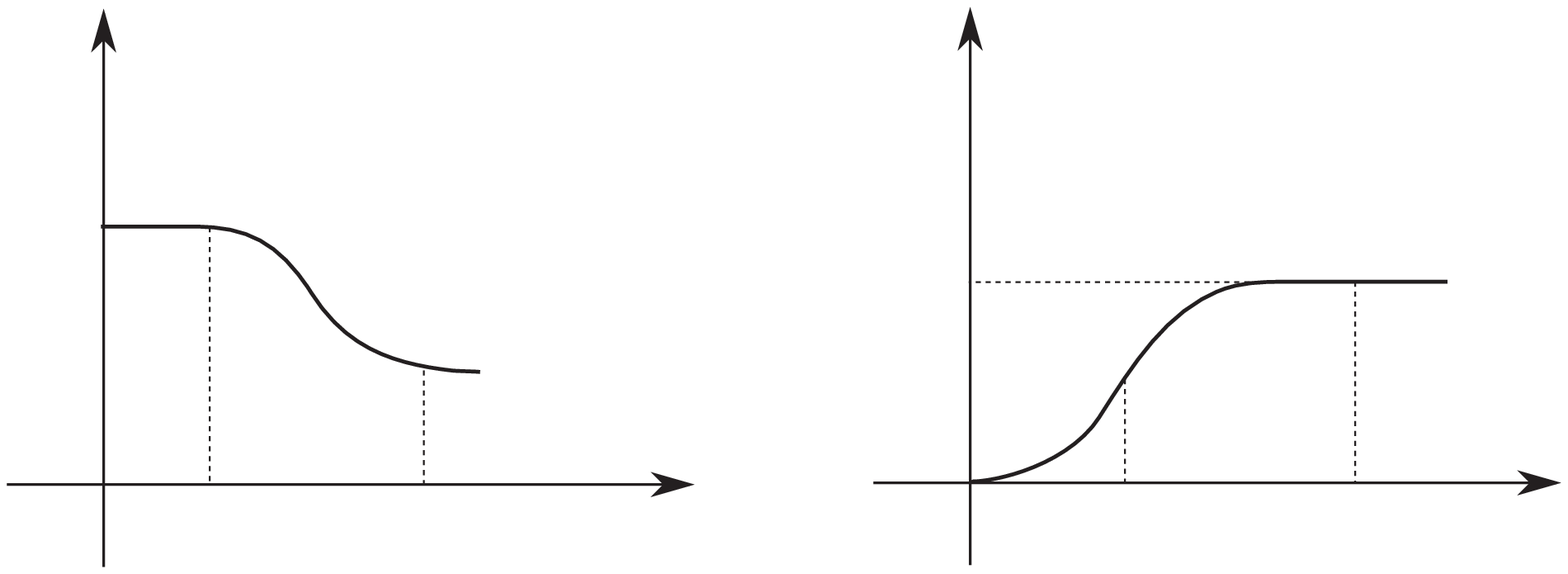}
	 \put(150,12){$r$}
	 \put(3,115){$f(r)$}
	 \put(43,10){$\epsilon/2$}
	 \put(95,10){$1$}
	 \put(13,75){$C_{0}$}
	 \put(196,115){$g(r)$}
	 \put(343,12){$r$}
	 \put(207,62){$K$}
	  \put(245,10){$\epsilon/2$}
	 \put(300,10){$1$}
		\end{overpic}
	\caption{}
	\label{fig: graph contact}
\end{figure}

The following lemma follows easily from the above construction. 

\begin{lemma}\label{lemma: open book}
Let $(V,\lambda; \phi)$ be an abstract open book. 
Suppose $\phi$ is the identity map on a (closed) collar neighborhood $\nu_{V} (\del V)$ of $\del V$. 
Then, $OB(V\setminus \mathring{\nu}_{V}(\del V), \lambda|; \phi|)$ is contactomorphic to 
$OB(V,\lambda; \phi)$, where $\lambda|$ and $\phi|$ are the restrictions of $\lambda$ and $\phi$ to 
$V \setminus \mathring{\nu}_{V}(\del V)$, respectively. 
\end{lemma}

\subsection{Strong symplectic fillings of contact manifolds}

Let $\pi: (E, \Omega) \rightarrow \D$ be a symplectic 
Lefschetz-Bott fibration over the closed unit disk $\D$ such that $\Omega$ is exact near the boundary $\del E$  and 
on each regular fiber of $\pi$. 
%Throughout this subsection, we assume \textit{triviality near the vertical boundary $\del_{v}E$} for $\pi: (E, \Omega) \rightarrow \D$, that is, on a collar neighborhood of $\del_{v}E$, $\pi$ is equivalent to the product map 
%$$
%id \times \pi: [1-\epsilon, 1] \times \del_{v}E  \rightarrow [1-\epsilon, 1] \times \del \D^2
%$$
%for some small $\epsilon>0$ with $\Omega$ identified with $\Omega_{v}+\omega_{b}$
We may assume that $\Omega$ is non-degenerate, i.e., symplectic on $E$ by \cite[Lemma 1.5]{Sei} and also that 
the Liouville vector field of $\lambda$, where $d\lambda=\Omega$, defined near $\del E$ points outwards along $\del E$. 
Horizontal triviality allows us to identify a collar neighborhood $\nu_{E}(\del_{h}E)$ of $\del_{h}E$ with 
$$
	([-2\epsilon, 0] \times \del V \times \D, de^{t}\lambda_{\del V}+d\lambda_{b})
$$
for some small $\epsilon>0$, where $V=\pi^{-1}(1)$, $\lambda_{\del V}=\lambda|_{T{\del V}}$ and $d\lambda_{b}$ is an exact symplectic form on $\D$. 
Moreover, for future use, identify it with 
$$
([-\epsilon, \epsilon] \times \del V' \times \D, de^{t}\lambda_{\del V'}+d\lambda_{b}), 
$$
where $V'=V \setminus (-\epsilon,0] \times \del V$. 
We deform the symplectic structure $\Omega$ so that 
the $1$-form $\lambda_{b}$ is a $1$-form on $\D$ which has the form 
$$h(s)d\varphi,$$  
where
the function $h(s)$ is equal to $Ks^2$ for some constant $K>0$ near $s=0$ and $Ke^{s-1}$ on $[1-\epsilon, \epsilon] \times \del \D \subset \D$, and $h'(s)>0$ except for $s=0$. 
Note that the deformation induces a one-parameter family of contact forms on $\del E$, 
so it does not affect the initial contact structure on $\del E$ up to isotopy.
By definition, 
in a neighborhood of the corners, the symplectic form $de^{t}\lambda_{\del V'}+d\lambda_{b}$ can be written as 
$$
	de^{t}\lambda_{\del V'}+Kde^s d\varphi. 
$$
Hence, the Liouville vector field of $\lambda$ coincides with $\del_{t}+\del_{s}$ on this neigborhood. 
To smooth the corners, first let $\gamma_{\epsilon}$ be the image of a curve 
$[-\epsilon, \epsilon] \rightarrow [-\epsilon, \epsilon] \times [1-\epsilon, 1]$, 
$r \mapsto (f(r), g(r))$ (see Figure \ref{fig: graph of corners}), where $f, g$ satisfies that 
\begin{itemize}
\item $f(r)=\epsilon$ and $g(r)=1+r$ for $r \leq -\epsilon/2$; 
\item $f'(r)<0$ and $g'(r)>0$ for $-\epsilon/2 < r < 0$;
\item $f(r)=-r$ and $g(r)=1$ for $r \geq 0$. 
\end{itemize}
Write $\Gamma_{\epsilon}$ for the component of $([-\epsilon, \epsilon] \times [1-\epsilon, 1])  \setminus \gamma_{\epsilon}$ containing 
the point $(\epsilon,1)$ (see Figure \ref{fig: graph of corners}). 
Define the manifold $E_{\epsilon}$ by 
$$
	E_{\epsilon} := E \setminus \Phi_{h}^{-1}(\{ (s,x,te^{2\pi i \varphi}) \in [-\epsilon, \epsilon] \times \del V' \times \D \mid 
	(s,t) \in \Gamma_{\epsilon}\}), 
$$
which will be a smooth manifold we desired. 
Here, $\Phi_{h}: \nu_{E}(\del_{h}E) \rightarrow [-\epsilon, \epsilon] \times \del V' \times \D$ 
is a trivialization of $\nu_{E}(\del_{h}E)$.

\begin{figure}[t]
	\centering
	\begin{overpic}[width=180pt,clip]{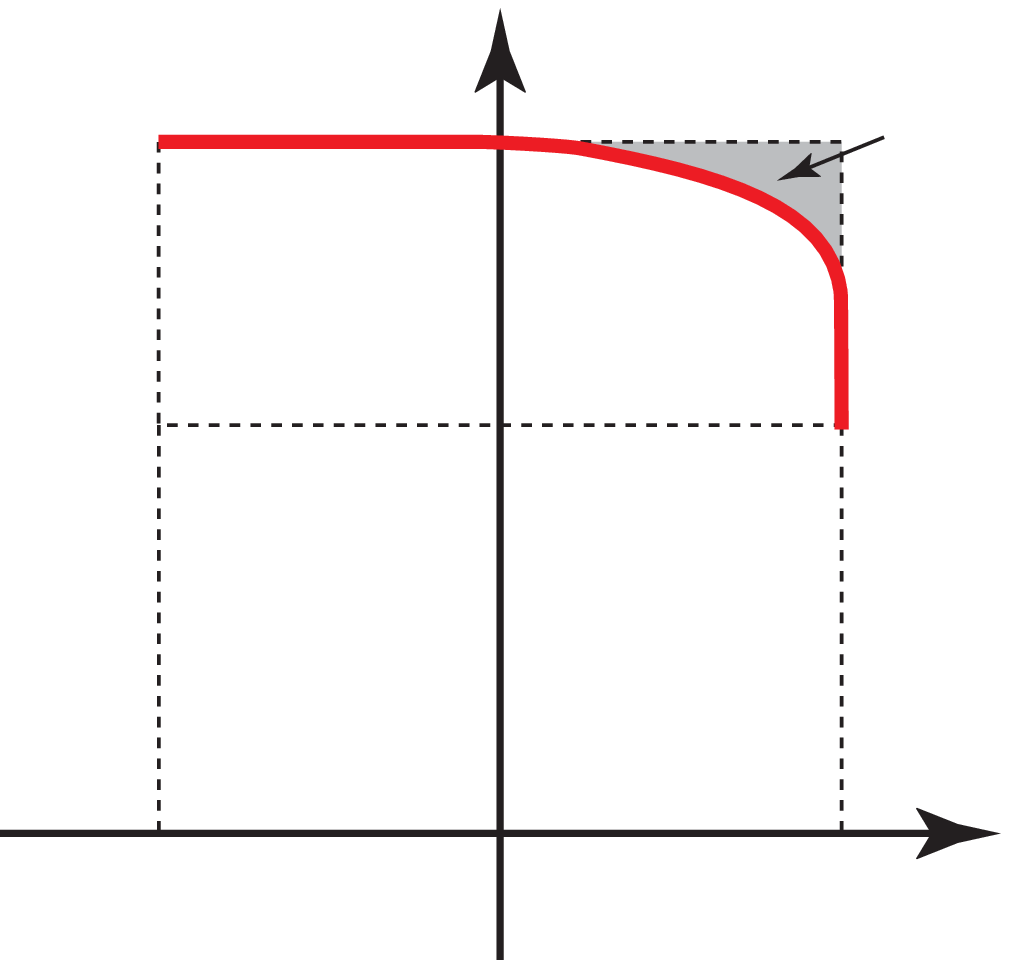}
	 %\linethickness{3pt}
	 \put(19,14){$-\epsilon$}
	 \put(148,14){$\epsilon$}
	 \put(65,87){$1-\epsilon$}
	 \put(82,137){$1$}
	 \put(160,143){$\Gamma_{\epsilon}$}
	 \put(40,152){$\gamma_{\epsilon}$}
		\end{overpic}
	\caption{}
	\label{fig: graph of corners}
\end{figure}

\begin{proposition}[{cf. \cite[Section 2.5]{LVHMW}}]\label{prop: strong fillings}
Let $\pi: (E, \Omega) \rightarrow \D$ be a symplectic Lefschetz-Bott fibration, $(V, d\lambda)=(\pi^{-1}(1), \Omega|_{\pi^{-1}(1)})$ and $\phi \in \Symp_{c}(V, d\lambda)$ the monodromy of $\pi$ along $\del \D$. 
Suppose $\Omega$ is non-degenerate on $E$ and exact near $\del E$ and on each regular fiber of $\pi$.  
Then, after deforming $\Omega$ and smoothing the corners of $E$, 
$E$ is a strong symplectic filling of the contact manifold $OB(V, \lambda; \phi)$. 
\end{proposition}

\begin{proof}
We may assume that $\Omega$ has been already deformed so that $\lambda_{b}$ agrees with $h(s)d\varphi$.  
Let $E_{\epsilon}$ be a smooth manifold constructed above. 
By construction, it has a convex boundary. 
Hence, it suffices to show that $(M=\del E_{\epsilon}, \ker(i_{X}\Omega|_{TM}))$ is contactomorphic to $OB(V, \lambda; \phi)$, where $X$ 
is the Liouville vector field defined near $\del E_{\epsilon}$. 
Set 
$$
M_{v}= M \cap \del_{v}E,\  
M_{h}=M \cap \del_{h} E 
$$
$$
M_{c}=\Phi_{h}^{-1}(\{ (s,x,te^{2\pi i \varphi}) \in [-\epsilon, \epsilon] \times \del V' \times \D\mid (s,t) \in \gamma_{\epsilon} \}).$$ 
Unless otherwise noted, we do not distinguish $\nu_{E}(\del_{h}E)$ from 
$[-\epsilon, \epsilon] \times \del V' \times \D$ in this proof.
Since the monodromy of the symplectic fibration $\pi|: M_{v} \rightarrow \del \D$ is $\phi$, 
it is isomorphic to 
$$
	V'(\phi)= (V' \times \R) / (\phi(x), \varphi+1) \sim (x, \varphi)
$$
equipped with the contact form $\lambda + d(H\varphi)$, where $H$ is a function with 
$\phi^{*}\lambda-\lambda=dH$ and $H \equiv 2\pi$ near $\del V'$. 
Define the contactomorphism $\Psi_{v}: V'_{H}(\phi) \rightarrow V'(\phi) \cong M_{v}$
by
$$
	\Psi_{v}: (x,\varphi) \mapsto (x, {\varphi}/H).
$$
Next, we define the diffeomorphism 
$\Psi_{h}: \del V' \times D_{1+\epsilon} \rightarrow M_{h} \cup M_{c} \subset [-\epsilon, \epsilon] \times \del V' \times \D$ 
by 
$$
	\Psi_{h}(x,r,\varphi) = 
	\begin{cases}
		 (\epsilon, x, re^{2\pi i \varphi}) & (r< 1-\epsilon), \\
		(f(r-1), x, g(r-1)e^{2\pi i \varphi}) & (r \geq 1-\epsilon). 
	\end{cases}
$$
A direct computation shows that  
$$
	\Psi_{h}^{*}(e^t \lambda_{\del V'} + h(r)d\varphi)= 
	\begin{cases}
	e^{\epsilon}\lambda|_{\del V'}+ h(r)d\varphi & (r<1-\epsilon), \\
	e^{f(r-1)}\lambda|_{\del V'}+h(g(r-1))d\varphi & (r \geq 1-\epsilon).
	\end{cases}
%	\begin{cases}
%		e^{\epsilon}\lambda|_{\del V'}+ Ks^2 d\varphi & (s<1-\epsilon) \\
%		e^{1-s}\lambda|_{\del V'}+ Kd\varphi & (s\geq 1-)
%	\end{cases}
$$
The pull-back $1$-form on $\del V' \times D^2(1+\epsilon)$ satisfies the conditions for them being contact in Section \ref{section: open book}.
As a consequence, the following diagram is commute and each map is a contactomorphim: 
\[
   \xymatrix{
       \ar[r]^-{\Phi} \del V'\times A(1,1+\epsilon) \ar[d]_{\Psi_{h}} & 
       [-\epsilon, 0] \times \del V' \times S^1 \subset V_{H}(\phi) \ar[d]^{\Psi_{v}} \\
      M_{h}\cap M_{c} \supset \Psi_{h}(\del V'\times A(1,1+\epsilon))  \ar[r]^-{Id} & \Psi_{v}([-\epsilon, 0] \times \del V' \times S^1) \subset M_{v} \cap M_{c}, 
   }
\]
where $\Phi$ is the map (\ref{eqn: glue}) defined in the previous subsection.
Hence we conclude that $(M, \ker(i_{X}\Omega|_{TM}))$ is contactomorphic to $OB(V', \lambda|_{V'}; \phi|_{V'})$, which is 
contactomorphic to $OB(V, \lambda; \phi)$ by Lemma \ref{lemma: open book}. 
This completes the proof.
\end{proof}

\begin{corollary}
Let $(V,\omega=d\lambda)$ be a Liouville domain and 
$\phi$ an element of $\Symp_{c}(V, \omega)$. 
Suppose there exists a collection of spherically fibered coisotropic submanifolds $(C_{1}, \ldots, C_{k})$ in $(V,\omega)$ such that 
$\phi$ is symplectically isotopic to the product $\tau_{C_{1}} \circ \cdots \circ \tau_{C_{k}}$ of fibered Dehn twists. 
Then, a contact structure compatible with the open book $(V, \lambda; \phi)$ is strongly fillable. 
\end{corollary}

\begin{proof}
One can associate to the collection $(C_{1}, \ldots, C_{k})$ in $(V,\omega)$ 
a symplectic Lefschetz-Bott fibration over $\D$ whose collection of vanishing cycles equals 
the given collection for some distinguished basis. 
The corollary straightforwardly follows from this fact coupled with Proposition \ref{prop: strong fillings}. 
\end{proof}

\begin{acknowledge}
The author would like to thank Tadashi Ashikaga, Joontae Kim, Myeonggi Kwon and Kaoru Ono for helpful discussions. 
\end{acknowledge}

%next line adds the Bibliography to the contents page
\addcontentsline{toc}{chapter}{Bibliography}
%uncomment next line to change bibliography name to references
%\renewcommand{\bibname}{References}
%\bibliography{refs}        %use a bibtex bibliography file refs.bib
%\bibliographystyle{abbrv}
\bibliographystyle{abbrv}
\bibliography{LineBundle}

\end{document}